\documentclass[reqno,11pt]{amsart}
\usepackage{amsfonts}
\usepackage{a4wide}	
\usepackage{bm}

\numberwithin{equation}{section}

\usepackage{amsmath,amssymb,amsthm,amsfonts}
\usepackage{mathrsfs}
\usepackage{bbm}
\usepackage{hyperref}
\hypersetup{colorlinks=true,linkcolor=blue,citecolor=green,urlcolor=blue}
\usepackage[norefs,nocites]{refcheck}
\usepackage{xcolor}
\usepackage{extarrows}

\allowdisplaybreaks[4]%
\expandafter\providecommand\csname subjclassname@2020\endcsname{2020 Mathematics Subject Classification}

\newtheorem{lemma}{Lemma}[section]
\newtheorem{theorem}{Theorem}[section]
\newtheorem{definition}{Definition}[section]
\newtheorem{proposition}{Proposition}[section]
\newtheorem{remark}{Remark}[section]

\renewcommand{\S}{\mathbb{S}}

\newcommand{\om}{\omega}

\newcommand{\eqdef}{\overset{\mbox{\tiny{def}}}{=}}

\newcommand{\norm}[1]{\left\Vert #1 \right\Vert}
\newcommand{\moc}[1]{\left\lvert #1 \right\rvert}

\date{}							

\everymath{\displaystyle}
\begin{document}
\raggedbottom
	\title[Relativistic collision geometry and near-vacuum solutions]{Relativistic collision geometry and global near-vacuum solutions at critical transverse regularity}

\author[X.-F. Hu]{Xinfeng Hu}
\address[XFH]{School of Mathematics Science, Fudan University, Shanghai 200433, P.R. ~China}
\email{xfhu25@m.fudan.edu.cn}

\author[S.-Q. Liu]{Shuangqian Liu}
\address[SQL]{School of Mathematics and Statistics, and Key Lab NAA-MOE, Central China Normal University, Wuhan 430079, P.R.~China}
\email{sqliu@ccnu.edu.cn}

\author[H.-R. Peng]{Haoran Peng}
\address[HRP]{School of Mathematics and Statistics, and Hubei Key Lab--Math. Sci., Central China Normal University, Wuhan 430079, P.R.~China}
\email{mathematicphr@mails.ccnu.edu.cn}

\author[Y. Zhou]{Yi Zhou}
\address[YZ]{School of Mathematics Science, Fudan University, Shanghai 200433, P.R. ~China}
\email{yizhou@fudan.edu.cn}

\date{}

\subjclass[2020]{Primary: 35Q20, 35B35, 35B36; Secondary: 76P05}	

\keywords{Relativistic Boltzmann equation, near vacuum, critical transverse regularity, transport interaction estimates, relativistic collision geometry.}

	\begin{abstract}
We establish global well-posedness near vacuum for the relativistic Boltzmann equation with soft and hard potentials in the range $\gamma\in(-2,1]$. The initial data have critical Besov regularity of order $2/r$ in the transverse spatial variables for an admissible range of $r$ depending on $\gamma$, together with weighted integrability in the longitudinal and momentum variables. The proof combines a fixed-direction div--curl lemma with bilinear collision estimates. In particular, to control the gain term, we combine localization to spherical caps with energy-weighted angular estimates. The key relativistic estimate controls the singularity associated with the first component of the post-collisional relative velocity. These bounds yield the required time decay.
\end{abstract}

\maketitle

	
\setcounter{tocdepth}{1}
\tableofcontents
\section{Introduction}

\subsection{Relativistic transport and the Cauchy problem}

A near-vacuum distribution describes a rare gas whose long-time dynamics are governed by the competition between free transport and quadratic collisions. For the relativistic Boltzmann equation, a basic question is whether global solutions can be constructed with limited spatial regularity and no momentum derivatives. We address this question in a weighted anisotropic space, imposing critical regularity only in two transverse spatial variables. With the rest mass and speed of light normalized to one, write

$$
p^0=\sqrt{1+|p|^2},
\qquad
\widehat p=\frac{p}{p^0},
\qquad
(t,x,p)\in[0,\infty)\times\mathbb R^3\times\mathbb R^3.
$$
Here $p$ is momentum, $p^0$ is energy, and $\widehat p$ is transport velocity. The particle distribution $f=f(t,x,p)$ is nonnegative and solves
\begin{equation}\label{equ:Rela-Boltzmann}
    \left\{
\begin{aligned}
\partial_t f+\widehat p\cdot\nabla_x f&=Q(f,f),\\
f(0,x,p)&=f_0(x,p).
\end{aligned}
\right.
\end{equation}
The bilinear collision operator is defined by
$$
\begin{aligned}
Q(f,h)(t,x,p)
={}&\int_{\mathbb R^3}\int_{\mathbb S^2}
B(p,q,\omega)
\bigl[
f(t,x,p')h(t,x,q')
-f(t,x,p)h(t,x,q)
\bigr]
\,d\omega\,dq\\
\eqdef&Q^+(f,h)(t,x,p)-Q^-(f,h)(t,x,p),
\end{aligned}
$$
where $Q^+$ and $Q^-$ are the gain and loss terms, respectively.

The collision geometry is expressed in center-of-momentum coordinates. Introduce
$$
g=\sqrt{2(p^0q^0-p\cdot q-1)},
\qquad
\mathfrak s=4+g^2,
\qquad
P=p+q,
\qquad
E=p^0+q^0.
$$
The post-collisional momenta and energies are given by
\begin{equation}\label{defi:p'}
    \begin{aligned}
p'&=\frac P2+\frac g2
\left(
\omega+
\frac{P(P\cdot\omega)}
{\sqrt{\mathfrak s}(E+\sqrt{\mathfrak s})}
\right),
&
q'&=P-p',\ \om\in\S^2,\\
p'^0&=\frac E2+
\frac{gP\cdot\omega}{2\sqrt{\mathfrak s}},
&
q'^0&=E-p'^0.
\end{aligned}
\end{equation}

This parametrization preserves momentum and energy:
$$
P = p+q=p'+q',
\qquad
E =  p^0+q^0=p'^0+q'^0.
$$
For $g>0$, the scattering angle $\theta\in[0,\pi]$ is the angle between the incoming and outgoing relative momenta in the center-of-momentum frame. It is defined invariantly by
\[
\cos\theta
=\frac{(p-q)\cdot(p'-q')-(p^0-q^0)(p'^0-q'^0)}{g^2};
\]
see \cite[Eq.~(1.4)]{Strain2010Vacuum}. On the null set $g=0$, we set $\theta=0$ and define the collision kernel to be zero; this convention does not affect the collision integrals.

For $g>0$, the kernel has the form
\begin{equation*}
B(p,q,\omega)
=\frac{g\sqrt{\mathfrak s}}{4p^0q^0}
\,g^{\gamma-1}\sigma_0(\theta),
\qquad
0\leq\sigma_0(\theta)\leq C_0,
\qquad
-2<\gamma\leq 1,
\end{equation*}
where $\sigma_0$ is measurable. Thus the differential cross section contains the factor $g^{\gamma-1}$, while the total power of $g$ in the collision kernel, apart from $\sqrt{\mathfrak s}/(4p^0q^0)$, is $\gamma$. All subsequent restrictions on $\gamma$ refer to this convention.

Since $Q$ is quadratic, its linearization at vacuum is zero and the linearized equation is free transport. Thus global control near vacuum must come from the accumulated interaction of particle trajectories. The relativistic velocity $\widehat p=p/p^0$ is bounded, and its derivative degenerates at high momentum. Consequently, momentum separation alone does not give a uniform measure of transport separation. This is the main obstruction to carrying classical low-regularity estimates directly into the relativistic setting.

Our purpose is to quantify that interaction in an integral norm, with polynomial localization along characteristics, while avoiding pointwise control over the full phase space. This permits rough dependence on the longitudinal and momentum variables. The gain term is decisive: one must control the separation of the outgoing trajectories even when their longitudinal velocities coincide. The estimates below show how the angular geometry and energy weights compensate for this degeneracy.

We fix the spatial decomposition
$$
x=(\eta,y),
\qquad
\eta=x_1\in\mathbb R,
\qquad
y=(x_2,x_3)\in\mathbb R^2,
$$
and resolve frequencies only in $y$. The same longitudinal direction is used for every momentum pair and every dyadic block. This choice specifies the solution space once and for all; the directional degeneracy that it leaves in the collision integral becomes a central object of the estimates.

\subsection{Assumptions and the main result}
\label{subsec:rel-main-results}
Let $r'=r/(r-1)$, and assume that
\begin{equation}\label{equ:gamma-r-1}
  \begin{cases}
\displaystyle
\max\left\{
\frac{3}{\gamma+2},\frac{2}{-\gamma}
\right\}<r<\infty,
& -2<\gamma<0,
\\[2ex]
2<r<\infty,
& 0\leq\gamma\leq1.
\end{cases}
\end{equation}
In both cases, choose $\ell$ and $m$ so that
\begin{equation}\label{equ:gamma-r-2}
  \ell>1+\frac{3}{r'},\qquad m>\frac{3\ell}{2}+\frac12+\frac{1}{2r'}.
\end{equation}
After fixing these parameters, choose
\begin{equation}\label{equ:vare}
0<\varepsilon<\min\left\{
\frac13, 2-r', 3+(\gamma-1)r', \moc{2-(\gamma+2)r'}
\right\}.
\end{equation}
For $\ell,m$ above, we use the time-dependent weight
$$
W(t,x,p)
=\langle x-(1+t)\widehat p\rangle^\ell(p^0)^m,
\qquad
W_0(x,p)=\langle x-\widehat p\rangle^\ell(p^0)^m,
$$
where $\langle z\rangle=(1+|z|^2)^{1/2}$. Let $\Delta_j$, $j\geq-1$, denote the nonhomogeneous Littlewood--Paley projections in $y$ defined in Section~\ref{sec:weighted-spaces}. Define the initial-data space by
$$
 \tilde L_{p,\eta}^r(B_{r,1,W_0}^{2/r}) ,
$$
where
$$
\|f\|_{ \tilde L_{p,\eta}^r(B_{r,1,W_0}^{2/r}) }
:=\sum_{j\geq-1}2^{2j/r}
\|W_0(x,p)\Delta_j f\|_{L^r_{x,p}}.
$$
For $T\in(0,\infty]$, write $W_p=W(t,x,p)$, define the energy and interaction functionals
$$
\begin{aligned}
\mathscr E_T(f)
&=\sum_{j\geq-1}2^{\frac{2j}{r}}
\|W\Delta_j f\|_{L^\infty(0,T;L^r_{x,p})},\\
\mathscr D_T(f,g)
&=\sum_{j,k\geq-1}2^{\frac{2(j+k)}{r}}
\norm{\moc{\widehat{p}_1 - \widehat{q}_1}^{\frac{1}{r}}\norm{W_p\Delta_j f(p)}_{L_y^{r}}\norm{W_q\Delta_k g(q)}_{L_y^{r}}}_{L_{t,p,q,\eta}^{r}},
\end{aligned}
$$
where the time interval in the last norm is $[0,T]$.
In these definitions, the weight is applied after the frequency projection.
We also write
$\mathscr D_T(f)=\mathscr D_T(f,f)$.
The initial energy used below is $\mathscr E_0(f_0)=\|f_0\|_{ \tilde L_{p,\eta}^r(B_{r,1,W_0}^{2/r}) }$.

Let
\[
U(t)a(x,p)=a(x-t\widehat p,p)
\]
be the free transport group.

\begin{theorem}
\label{thm:rel-main}
Let $\gamma,r,\ell,m$ satisfy \eqref{equ:gamma-r-1} and
\eqref{equ:gamma-r-2}. There is a constant $\delta_*>0$, depending only
on these parameters and the angular bound, such that if $f_0\geq0$ almost
everywhere and $\mathscr E_0(f_0)\leq\delta_*$, then
\eqref{equ:Rela-Boltzmann} has a global nonnegative mild solution $f$ with
\[
 \mathscr E_T(f)+\mathscr D^{1/2}_T(f)
 \leq C\mathscr E_0(f_0),\qquad T\geq0.
\]
Its free-transport pullback $U(-t)f(t)$ is continuous in
$\widetilde L^r_{p,\eta}(B^{2/r}_{r,1,W_0})$ and has initial value $f_0$.
The solution is unique among mild solutions with this initial trace for
which $\mathscr E_T(f)$ and $\mathscr D_T(f)$ are finite on every finite
interval. Here \emph{mild} means
\[
 f(t)=U(t)f_0+\int_0^t U(t-s)Q(f,f)(s)\,ds,
\]
with the integral well defined in the weighted space after free-transport
pullback.
\end{theorem}
We make two remarks on Theorem \ref{thm:rel-main}.
\begin{remark}
The criticality in this theorem refers to the Besov embedding
$$
B^{2/r}_{r,1}(\mathbb R_y^2)
\hookrightarrow L^\infty(\mathbb R_y^2).
$$
The initial-data norm requires no derivatives in $\eta$ or $p$ and permits unbounded data. For example, let $a,b,c$ be nonnegative smooth compactly supported functions of $\eta,y,p$, respectively, with $a=1$ near zero and $b,c$ nonzero. For $0<\alpha<1/r$ and $\lambda>0$, set
\[
 f_0(\eta,y,p)=\lambda|\eta|^{-\alpha}a(\eta)b(y)c(p).
\]
Since the projections act only on $b$ and $W_0\lesssim\langle y\rangle^\ell$ on the supports of $a$ and $c$,
\[
 \mathscr E_0(f_0)
 \lesssim\lambda\||\eta|^{-\alpha}a\|_{L^r_\eta}\|c\|_{L^r_p}
 \sum_{j\geq-1}2^{2j/r}
 \|\langle y\rangle^\ell\Delta_j b\|_{L^r_y}
 \leq C_\alpha\lambda,
 \qquad \|f_0\|_{L^\infty_{x,p}}=\infty.
\]
Thus smallness in the theorem does not require uniform smallness of the density. Conversely, a weighted supremum bound alone does not imply the transverse Besov regularity required here.

For each fixed $\gamma\in(-2,1]$, the parameter conditions yield a corresponding range of admissible exponents $r$. The endpoint $\gamma=-2$ and estimates uniform as $\gamma\downarrow-2$ are not included.
\end{remark}

\begin{remark}
Glassey \cite{Glassey2006Vacuum} constructed global mild solutions to the relativistic Boltzmann equation in a weighted supremum framework. Here, we obtain global control under a smallness condition in an integral norm that allows unbounded phase-space densities. Our framework combines polynomial characteristic weights with critical Besov regularity in the two transverse spatial variables and weighted integrability in the longitudinal and momentum variables.
\end{remark}

\subsection{Relation to the classical problem}

For the classical Boltzmann equation, \cite{HLPZ2026} establishes near-vacuum well-posedness in weighted anisotropic Besov spaces using a bilinear collision estimate and a div--curl lemma. The present work uses the same transverse regularity principle. Its additional task is to control the nonlinear momentum-to-velocity map and the outgoing relativistic collision geometry. In a fixed longitudinal direction, $|\widehat p_1-\widehat q_1|$ may vanish for distinct momenta; the gain estimate must control the corresponding singularity for $p',q'$ as well.

The principal new estimate links relativistic angular integration to the fixed-direction transport interaction functional. It closes the global argument in the anisotropic integral space without adding longitudinal or momentum derivatives. The critical transverse exponent is the endpoint for the Besov embedding stated above; it is not asserted to be an optimal regularity threshold for the full equation.

\subsection{Ideas and strategies}

The proof consists of two parts. The collision estimates capture the relativistic features of the equation, while the transport estimates and iteration scheme follow the anisotropic framework of \cite{HLPZ2026}. This separation clarifies precisely where the geometry of the present equation \eqref{equ:Rela-Boltzmann} enters the analysis.

\medskip
\noindent\textit{Localization compatible with transport.}
The characteristic weight satisfies
$$
(\partial_t+\widehat p\cdot\nabla_x)W=0.
$$
Thus, the spatial weight is transported along free trajectories without generating any error terms, while its energy factor provides control at large momenta. Weighted frequency projections and Bernstein estimates ensure that this localization is compatible with transverse product estimates at regularity index $2/r$.

\medskip
\noindent\textit{The outgoing angular degeneracy.}
The loss estimate combines characteristic localization with a momentum-to-velocity change of variables. The gain term requires additional control of outgoing energies and outgoing longitudinal velocities. With
$$
L=Ee_1-\frac{P_1P}{E+\sqrt{\mathfrak s}},
\qquad
e_1=(1,0,0),
$$
the center-of-momentum representation yields
$$
\moc{\widehat  p'_1 - \widehat q'_1}
=\frac{g|L\cdot\omega|}{2p'^0q'^0},
\qquad
|L|^2=E^2-P_1^2\geq\mathfrak s.
$$
For $g>0$, the first identity locates the zero of the outgoing directional factor on a great circle of the collision sphere. The second supplies a lower bound for $|L|$, leaving the outgoing energy factors explicit. Spherical-cap estimates handle the resulting angular singularity together with the sets on which outgoing energy or characteristic localization is small. Interpolation and a decomposition of the incoming momentum regimes yield the time-integrable residual gain bound. This is the main geometric part of the proof.

\medskip
\noindent\textit{From collision geometry to the source bound.}
The residual gain and loss estimates feed into H\"older's inequality in the collision variables. For the gain term, the exchange of incoming and outgoing variables uses the invariant mass-shell measure. The remaining transverse products are treated by Bony's decomposition and weighted Bernstein inequalities, giving Proposition~\ref{prop:non-esti}. Thus the nonlinear source is bounded by the same fixed-direction interaction functional that appears in the transport estimate.


\subsection{Related works}

For the classical near-vacuum problem, Illner and Shinbrot \cite{is-84} constructed global solutions for a sufficiently rare gas, and Ha, Yamazaki, and Yun \cite{hyy-08} proved uniform $L^p$ stability. Ars\'enio \cite{ad-11} used transport dispersion for small-data existence under restrictions on the collision kernel. He and Jiang developed this approach in \cite{hj-17} and supplied the required loss-term estimates in \cite{hj-23}, obtaining global existence and scattering for three-dimensional cutoff soft potentials. Chen, Denlinger, and Pavlovi\'c \cite{cdp-21} treated the two-dimensional constant-kernel equation by bilinear spacetime estimates. Further gain-term estimates and small-data applications were obtained by He, Jiang, Kuo, and Liang \cite{hjkl-24}.

Vacuum problems with external forces and self-consistent fields have also been studied extensively. Guo \cite{guo-vpb-01} constructed global classical solutions to the Vlasov--Poisson--Boltzmann system near vacuum. Duan, Yang, and Zhu proved global existence for the Boltzmann equation with external forces \cite{dyz-05} and for the forced equation and Vlasov--Poisson--Boltzmann system in infinite vacuum \cite{dyz-06}. For the latter system, Chae, Ha, and Hwang studied large-time behavior \cite{chh-jde}, while Duan, Zhang, and Zhu established $L^1$ stability \cite{dzz-06}. Near-vacuum results for the Vlasov--Yukawa--Boltzmann system include asymptotic dynamics in the cutoff case \cite{ch-15} and stability without angular cutoff \cite{hxz-25}.

Low regularity is a central issue in this theory. Chen, Shen, and Zhang \cite{csz-24} identified the local Sobolev well/ill-posedness threshold for cutoff soft potentials; their global existence and scattering result \cite{csz-23} uses smallness in a scaling-critical weighted norm. Chaturvedi \cite{cs-21} proved vacuum stability for non-cutoff moderately soft potentials. Related vacuum results for the Landau and quantum Boltzmann equations were obtained by Luk \cite{jj-19} and Ouyang and Wu \cite{ow-22}, respectively. The anisotropic Besov result \cite{HLPZ2026} places regularity in the transverse variables and uses a div--curl interaction estimate. These works motivate an integral theory that separates transverse regularity from longitudinal and momentum integrability in the relativistic setting. Low-regularity theories for kinetic equations have also been developed near Maxwellian equilibria; see, for instance, \cite{dlss-21,dlx-16}.


For the relativistic equation, Glassey \cite{Glassey2006Vacuum} established global near-vacuum mild solutions with weighted $L^\infty$ control, without Sobolev differentiability assumptions. Strain \cite{Strain2010Vacuum} obtained global solutions uniformly in the speed of light and a Newtonian limit on finite time intervals. Ha, Jeong, and Strain \cite{HaJeongStrain2013} proved uniform $L^1$ stability for small, exponentially decaying data using a collision potential. Takou and Ciake \cite{tc-18} constructed global classical near-vacuum solutions for hard potentials in an expanding Robertson--Walker space-time. For the coupled relativistic Vlasov--Maxwell--Boltzmann system, Cao and Li \cite{CaoLi2026Vacuum} established vacuum stability by vector-field estimates. The present result concerns the field-free equation with a bounded angular factor and polynomial characteristic weights. It develops global control at critical transverse regularity through explicit estimates for the outgoing directional singularity.


\subsection{Other notations}
Throughout the paper, $C$ denotes a positive constant whose value may change from line to line. Unless otherwise specified locally, $i,j,k$ are dyadic indices. The integrability exponent and its conjugate are $r$ and $r'$, as in Section~\ref{subsec:rel-main-results}. For nonnegative quantities $A$ and $B$, we write $A\lesssim B$ if $A\leq CB$, and $A\thicksim B$ if both $A\lesssim B$ and $B\lesssim A$. We use $\langle z\rangle=(1+|z|^2)^{1/2}$ for scalars and vectors in any Euclidean dimension.

The blocks $\Delta_j$ always act in the transverse variables $y$. For a momentum $a\in\mathbb R^3$, we use the shorthand
\[
 W_a=W(t,x,a),\qquad W_0(a)=W_0(x,a),\qquad
 S_a=\langle x-(1+t)\widehat a\rangle.
\]
In particular, $\widehat p_1$ and $\widehat p'_1$ denote the first components of the incoming and outgoing velocities, respectively.

\subsection{Organization of the paper}

Section~\ref{sec:weighted-spaces} introduces the weighted transverse Besov spaces. Section~\ref{sec:decay-estimates} establishes the gain and loss decay bounds, which are used in Section~\ref{sec:collision-bound} to prove the nonlinear source estimate. Sections~\ref{sec:interaction-estimates} and~\ref{sec:global-existence} give the interaction estimate and global construction. The appendix records the collision identities and analytic tools.

\section{Littlewood--Paley decomposition and weighted spaces}\label{sec:weighted-spaces}

         Only the two transverse variables enter the frequency decomposition. This section introduces the decomposition and records the Besov and Chemin--Lerner norms used in the collision estimate. The underlying Littlewood--Paley theory is standard; see \cite{BCD}.

\subsection{Decomposition in the variables $y$}

The Fourier transform is taken only in $y$; the variables $(t,\eta,p)$ remain parameters. For $f(t,\eta,\cdot,p)\in\mathcal S(\mathbb R^2_y)$, set
        $$ \widehat{f}(t,\eta,z,p):=\int_{\mathbb{R}^2} e^{-iy\cdot z}f(t,\eta,y,p)dy.$$

Choose a smooth dyadic partition $(\varphi,\chi)$ taking values in $[0,1]$, with $\operatorname{supp}\varphi\subset\{z:3/4\leq|z|\leq8/3\}$ and $\operatorname{supp}\chi\subset\{z:|z|\leq4/3\}$, such that
        $$ \chi(z)+\sum_{j\geq 0} \varphi(2^{-j}z)=1, \quad\quad\forall z\in \mathbb{R}^2,$$
        and
        $$ \sum_{j\in \mathbb{Z}}\varphi(2^{-j}z) =1, \quad\quad \forall z\in\mathbb{R}^{2}- \{0\}.$$
        The associated nonhomogeneous blocks of $f$ are
\[
 \Delta_{-1}f(y)=\widetilde\psi*f(y),\qquad
 \Delta_jf(y)=2^{2j}\int_{\mathbb R^2}\psi(2^jz)f(y-z)\,dz,
 \quad j\geq0,
\]
where $\widetilde\psi=\mathcal F^{-1}\chi$, $\psi=\mathcal F^{-1}\varphi$,
and the convolution is in $y$. Set $\Delta_j=0$ for $j\leq-2$ and
\[
 S_jf=\sum_{-1\leq k\leq j-1}\Delta_kf\quad(j\geq0),
 \qquad S_jf=0\quad(j\leq-1).
\]
Thus $S_0=\Delta_{-1}$, and the kernel representation is
\[
 S_jf=\chi(2^{-j}D)f
 =2^{2j}\int_{\mathbb R^2}\widetilde\psi(2^jz)f(y-z)\,dz,
 \qquad j\geq0.
\]

        These operators resolve $f$ through the nonhomogeneous decomposition
        $$f = \sum_{j \geq -1} \Delta_j f.$$ The blocks $\Delta_j$, $j\geq -1,$ determine the function spaces below.

\subsection{Mixed norms and the endpoint Besov embedding}

      \begin{definition}
Let $1 \leq r,q,\bar{q} \leq \infty$ and $s \in \mathbb{R}$.

The nonhomogeneous Besov space is defined by
\[
B_{r,q}^s := \left\{ g \in \mathcal{S}'(\mathbb{R}^2_y) :
g = \sum_{j \ge -1} \Delta_j g,\ \|g\|_{B_{r,q}^s} < \infty \right\},
\]
where $\mathcal{S}'(\mathbb{R}^2_y)$ denotes the space of Schwartz distributions on $\mathbb{R}^2_y$.

The nonhomogeneous anisotropic Besov space is defined by
\[
\tilde{L}_{p,\eta}^r(B_{r,q}^s)
:= \left\{ f \in \mathcal{S}'(\mathbb{R}^3_p \times \mathbb{R}_\eta \times \mathbb{R}^2_y) :
f = \sum_{j \ge -1} \Delta_j f,\
\|f\|_{\tilde{L}_{p,\eta}^r(B_{r,q}^s)} < \infty \right\}.
\]

We also define the time-dependent Chemin--Lerner type space
\[
\tilde{L}_T^{\bar{q}} \tilde{L}_{p,\eta}^r(B_{r,q}^s)
:= \left\{ f \in \mathcal{S}'([0,T] \times \mathbb{R}^3_p \times \mathbb{R}_\eta \times \mathbb{R}^2_y) :
f = \sum_{j \ge -1} \Delta_j f,\
\|f\|_{\tilde{L}_T^{\bar{q}} \tilde{L}_{p,\eta}^r(B_{r,q}^s)} < \infty \right\}.
\]

The corresponding norms are given by
\[
\|g\|_{B_{r,q}^s}
= \left( \sum_{j \ge -1} \big( 2^{js} \|\Delta_j g\|_{L_y^r} \big)^q \right)^{\frac{1}{q}},
\]
\[
\|f\|_{\tilde{L}_{p,\eta}^r(B_{r,q}^s)}
= \left( \sum_{j \ge -1} \big( 2^{js} \|\Delta_j f\|_{L_{p,\eta,y}^r} \big)^q \right)^{\frac{1}{q}},
\]
and
\[
\|f\|_{\tilde{L}_T^{\bar{q}} \tilde{L}_{p,\eta}^r(B_{r,q}^s)}
= \left( \sum_{j \ge -1} \big( 2^{js} \|\Delta_j f\|_{L_t^{\bar{q}} L_{p,\eta,y}^r} \big)^q \right)^{\frac{1}{q}}, \ t\in[0,T],
\]
with the usual modifications when an exponent is $\infty$.

\medskip

Moreover, for the time-dependent weight
\[
W(t,x,p) = \langle x-(t+1)\widehat{p} \rangle^\ell(p^0)^m,
\]
we define the weighted anisotropic Besov norm by
\[
\|f\|_{\tilde{L}_T^{\bar{q}} \tilde{L}_{p,\eta}^r(B_{r,q,W}^s)}
= \left( \sum_{j \ge -1} \big( 2^{js} \| W(t,x,p)\Delta_j f \|_{L_t^{\bar{q}} L_{p,\eta,y}^r} \big)^q \right)^{\frac{1}{q}},\ t\in[0,T],
\]
again with the usual modifications when an exponent is $\infty$. Throughout the paper, we use $L^{r,W}$ to denote the weighted $L^r$ space.

\medskip

In particular, the case $q=1$ and $s=2/r$ (with $1 \le r < \infty$) plays a central role in this paper. We will mainly work in the spaces
\[
\tilde{L}_{p,\eta}^r(B_{r,1,W_0}^{2/r})
\quad \text{and} \quad
\tilde{L}_T^\infty \tilde{L}_{p,\eta}^r(B_{r,1,W}^{2/r}).
\]
\end{definition}

The product estimate needed later rests on an endpoint Besov embedding. For fixed $(t,p,\eta)$, $$ B_{r,1,W}^{2/r}\hookrightarrow L_{y}^{\infty,W},$$
The form used here is
\begin{equation}\label{Embedding inequailty}
          \norm{W_pf}_{L_y^{\infty}}\leq C \sum_{j\geq
        -1} 2^{\frac{2j}{r}}\norm{W_p\Delta_j f}_{L_{y}^r}.
        \end{equation}
        The proof is included in Section~\ref{sec-app}. 

\section{Decay estimates for collision factors}\label{sec:decay-estimates}
This section establishes time-decay bounds for the weighted factors
arising in the loss and gain terms of the collision operator. These bounds
will be used in the nonlinear estimate of Section~\ref{sec:collision-bound}.
Throughout, the parameters $\gamma,r,\ell,m,\varepsilon$ satisfy
\eqref{equ:gamma-r-1}, \eqref{equ:gamma-r-2}, and \eqref{equ:vare}.

For $t\geq0$ and $x,p\in\mathbb R^3$, define the nonnegative loss factor by
\[
 I_{\mathrm{loss}}(t,x,p)
 :=\left(
 \int_{\mathbb R^3\times\mathbb S^2}
 \frac{B(p,q,\omega)^{r'}}
 {W_q^{r'}|\widehat p_1-\widehat q_1|^{r'/r}}
 \,dq\,d\omega
 \right)^{1/r'}.
\]
The nonnegative gain factor is defined by
\begin{equation}\label{I-gain}
 I_{\mathrm{gain}}(t,x,p):=
\begin{cases}
\displaystyle
\left(
\int_{\mathbb{R}^3\times\mathbb S^2}
\dfrac{B(p,q,\omega)W_p^{r'}}
{W_{p'}^{r'}W_{q'}^{r'}
\moc{\widehat p'_1-\widehat q'_1}^{\frac{r'}r}}
\,dq\,d\omega
\right)^{1/r'},
& \mbox{if }\gamma\in[0,1],
\\[2ex]
\displaystyle
\left(
\int_{\mathbb{R}^3\times\mathbb S^2}
\dfrac{B^{r'}(p,q,\omega)W_p^{r'}(p^0q^0)^{\frac{r'}r}}
{W_{p'}^{r'}W_{q'}^{r'}
\moc{\widehat p'_1-\widehat q'_1}^{\frac{r'}r}}
\,dq\,d\omega
\right)^{1/r'},
& \mbox{if }\gamma\in(-2,0).
\end{cases}
\end{equation}
Here $W_a=W(t,x,a)$, and $p',q'$ are the post-collisional momenta.
The loss bound follows from a change from momentum to relativistic
velocity. The gain bound requires angular estimates for the outgoing
velocities and weights.

\begin{proposition}[Decay of the loss factor]\label{prop:Estimate for I-loss}
Assume that $\gamma,r$ satisfy \eqref{equ:gamma-r-1}. Then
\[
 I_{\mathrm{loss}}^{r'}(t,x,p)
 \lesssim(1+t)^{-(3+\min\{\gamma,0\}r'-r'/r)},
\]
uniformly for $t\geq0$ and $x,p\in\mathbb R^3$.
\end{proposition}

    \begin{proof} We first consider $\gamma\in(-2,0)$.

The kernel bound \eqref{formu:estimate for B} gives
        $$ I_{\mathrm{loss}}^{r'} \lesssim \int_{\mathbb{R}^3 \times \mathbb{S}^2} \dfrac{ 1+\moc{p-q}^{\gamma r'}}{W_{q}^{r'}\moc{\widehat{p}_1 - \widehat{q}_1}^{\frac{r'}{r}}} dqd\omega. $$
Introduce the change of variables
$b=\widehat q$, so that $\moc b<1$ and
$$ D_q b = \dfrac{I}{q^0} - \dfrac{q\otimes q}{(q^0)^3},$$ whose eigenvalues are $ (q^0)^{-1},(q^0)^{-1},(q^0)^{-3},$ and hence $dq = (q^0)^5db.$ Since the map $q\mapsto\widehat q$ is globally Lipschitz, we have
$\moc{\widehat p-\widehat q}\lesssim\moc{p-q}$.
The weight controls the Jacobian when $mr'\geq 5$, so
        \begin{align*}
       \notag  I_{\mathrm{loss}}^{r'}  &\lesssim \int_{\{|b|< 1\} \times \mathbb{S}^2} \dfrac{(1+\moc{p-q}^{\gamma r'})(q^0)^5}{(q^0)^{mr'}\langle x-(t+1)b\rangle^{\ell r'}\moc{\widehat{p}_1 - \widehat{q}_1}^{\frac{r'}{r}}} dbd\omega  \\
           &\lesssim \int_{\{|b|< 1\} \times \mathbb{S}^2} \dfrac{(1+\moc{\widehat{p}-\widehat{q}}^{\gamma r'})(q^0)^{5-mr'}}{\langle x-(t+1)b\rangle^{\ell r'}\moc{\widehat{p}_1 - b_1}^{\frac{r'}{r}}} dbd\omega  \\
        \notag   &\lesssim \int_{\mathbb{R}^3} \dfrac{(1+\moc{\widehat{p}- b }^{\gamma r'})}{\langle x-(t+1)b\rangle^{\ell r'}\moc{\widehat{p}_1 - b_1}^{\frac{r'}{r}}} db .
        \end{align*}
Next, rescale by $z= (1+t)(b-\widehat{p})$ and write $S=x - (1+t)\widehat{p}$. We therefore obtain
        \begin{align*}
       \notag  I_{\mathrm{loss}}^{r'}   &\lesssim  \int_{\mathbb{R}^3}  \dfrac{(1+t)^{-3}(1+(1+t)^{-\gamma r'}\moc{z}^{\gamma r'})}{(1+t)^{-\frac{r'}{r}}\moc{z_1}^{\frac{r'}{r} }\langle S - z\rangle^{\ell r'}} dz\\
           &\lesssim (1+t)^{-3+\frac{r'}{r}} \int_{\mathbb{R}^3} \dfrac{dz}{\moc{z_1}^{\frac{r'}{r}}\langle S-z\rangle^{\ell r'}  } + (1+t)^{-3-\gamma r' + \frac{r'}{r}}\int_{\mathbb{R}^3} \dfrac{\moc{z}^{\gamma r'}dz}{\moc{z_1}^{\frac{r'}{r}}\langle S-z\rangle^{\ell r'} }\\
         \notag  &:= (1+t)^{-3+\frac{r'}{r}}I_1 + (1+t)^{-3-\gamma r' + \frac{r'}{r}}I_2\lesssim  (1+t)^{-3-\min\{\gamma,0\} r' + \frac{r'}{r}}(I_1 + I_2).
        \end{align*}
For $I_1,$ integrating first with respect to $y=(z_2,z_3)$,
we obtain
        \begin{align*}
        \notag   I_1&= \int_{\mathbb{R}}\dfrac{dz_1}{\moc{z_1}^{\frac{r'}{r}}}\int_{\mathbb{R}^2} \dfrac{dy}{\left( 1+ \moc{S_1 - z_1}^2 + \moc{S_y - y}^2 \right)^{\frac{\ell r'}{2}}}\\
            &= \int_{\mathbb{R}}\dfrac{dz_1}{\moc{z_1}^{\frac{r'}{r}}}\int_{[0,\infty)\times [0,2\pi]} \dfrac{\rho d\rho d\theta}{\left( 1+ \moc{S_1 - z_1}^2 +\rho^2 \right)^{\frac{\ell r'}{2}}}
       \lesssim\int_{\mathbb{R}} \dfrac{dz_1}{\moc{z_1}^{\frac{r'}{r}}\langle S_1 - z_1 \rangle^{\ell r' - 2}}
        \lesssim 1,
        \end{align*}
the final bound follows from Lemma \ref{Integral Estimate Lemma}.
For $I_2,$ separate the neighborhood of the singularity from its complement:
        \begin{align*}
      I_2&=\int_{\mathbb{R}^3} \dfrac{\moc{z}^{\gamma r'}dz}{\moc{z_1}^{\frac{r'}{r}}\langle S-z\rangle^{\ell r'}}
             =\left(\int_{\moc{z}\leq 1} + \int_{\moc{z}\geq 1}\right)\dfrac{\moc{z}^{\gamma r'}dz}{\moc{z_1}^{\frac{r'}{r}}\langle S-z\rangle^{\ell r'}}
             := I_{2,1} + I_{2,2}.
        \end{align*}
On the complement, $I_{2,2} \leq  \int_{\mathbb{R}^3}\dfrac{dz}{\moc{z_1}^{\frac{r'}{r}}\langle S-z\rangle^{\ell r'}} :=I_1 \lesssim 1,$
whereas near the singularity,
        \begin{align*}
        I_{2,1}\lesssim \int_{\moc{z}\leq 1} \dfrac{dz}{\moc{z_1}^{\frac{r'}{r}}\moc{z}^{-\gamma r'}}
                   \lesssim& \int_{\moc{z_1}\leq 1}\int_{\moc{y}\leq 1} \dfrac{dz_1dy}{\moc{z_1}^{\frac{r'}{r}}(|y|^2 + |z_1|^2)^{\frac{-\gamma r'}{2}}}\\
                   \lesssim& \int_{\moc{z_1}\leq 1} \left( \dfrac{1}{|z_1|^{\frac{r'}{r} -\gamma r'-2}} + \dfrac{1}{|z_1|^{\frac{r'}{r}}} \right)dz_1 \lesssim 1,
        \end{align*}
both rescaled integrals are therefore bounded.

We next consider $\gamma\in[0,1]$.

By \eqref{formu:estimate for B} and the same calculation as for $\gamma\in(-2,0)$,
using $mr'\geq 5$, we obtain
$$
I_{\mathrm{loss}}^{r'}
\lesssim
\int_{\mathbb{R}^3}
\dfrac{dq}{
\moc{\widehat{p}_1-\widehat{q}_1}^{\frac{r'}{r}}
\langle x-(1+t)\widehat{q}\rangle^{\ell r'}
(q^0)^{mr'}
}\lesssim
(1+t)^{-3+\frac{r'}{r}}I_1
\lesssim
(1+t)^{-3-\min\{\gamma,0\} +\frac{r'}{r}},
$$
where the final inequality follows from the preceding bound for $I_1$.
This completes the proof.
    \end{proof}


For $c\in\mathbb{S}^2$ and $\rho\geq0$, we denote the spherical cap
centered at $c$ with radius $\rho$ by
$$
\mathcal{M}(c,\rho)
=
\{\omega\in\mathbb{S}^2:|\omega-c|\leq\rho\}.
$$
We also write
$$
\mathcal{M}(\pm c,\rho)
=
\mathcal{M}(c,\rho)\cup\mathcal{M}(-c,\rho).
$$
We first estimate the singular angular integral over a spherical cap.
We then combine this estimate with the characteristic weights to obtain
the time decay of $I_{\mathrm{gain}}$.

The required angular bound is uniform in both the cap center and the singular direction.
    \begin{lemma}\label{lem:Local angular integration lemma}
Fix $e,c\in \mathbb{S}^2 $ and $\alpha\in (-1,0].$ For all $\rho\geq0,$   then
      \begin{equation}\label{Estimate for Angular integration}
        \int_{\mathcal{M}(c,\rho)} \moc{e\cdot \omega}^{\alpha} d\omega\lesssim \rho^{2+\alpha},
      \end{equation}
where the implicit constant is uniform in $e,c$ and $\rho$.
    \end{lemma}

    \begin{proof}
    The case $\alpha=0$ is immediate, so it suffices to consider $\alpha\in(-1,0)$.
By rotation, we may assume $e=(0,0,1)$. Thus
      \begin{equation*}
        \int_{\mathcal{M}(c,\rho)}\moc{e\cdot \omega}^{\alpha} d\omega = \int_{\mathcal{M}(c,\rho)} \moc{\omega_3}^{\alpha}d\omega.
      \end{equation*}
We consider separately the cases $\rho\geq\frac{1}{8}$ and
$0\leq\rho<\frac{1}{8}$.

For $\rho\geq\frac18$, integration over the whole sphere gives
      \begin{align}\label{formu:angular-estimate-2}
        \int_{\mathcal{M}(c,\rho)}\moc{\omega_3}^{\alpha}d\omega&\lesssim \int_{\mathbb{S}^2}\moc{\omega_3}^{\alpha} d\omega
                                                                \lesssim \int_{[0,2\pi]\times [0,\pi]} \moc{\cos \varphi}^{\alpha}d\varphi d\theta
                                                               \lesssim 1 \lesssim \rho^{2+\alpha}.
      \end{align}

For $0\leq\rho<\frac18$, we distinguish two cases according to $|c_3|$.
If $\moc{c_3} > 2\rho,$ then $ \moc{\omega_3}\geq \moc{c_3} - \moc{\omega_3  - c_3}\geq \moc{c_3} - \moc{\omega-c }> \rho.$ Consequently,
         \begin{equation}\label{formu:angular-estimate-3}
           \int_{\mathcal{M}(c,\rho)} \moc{\omega_3}^{\alpha} d\omega \leq \moc{\mathcal{M}(c,\rho)}\rho^{\alpha},
         \end{equation}
The cap-area bound $|\mathcal{M}(c,\rho)| \lesssim \rho^2$
then proves
        \eqref{Estimate for Angular integration}.

If $\moc{c_3} \leq 2\rho,$ then $\moc{\omega_3}\leq \moc{c_3} + \moc{\omega-c}<3\rho.$ Rotating about the third axis, we may assume that $c$ has azimuthal
angle zero. In spherical coordinates, we then obtain
        \begin{align}\label{formu:angular-estimate-4}
         \int_{\mathcal{M}(c,\rho)} \moc{\omega_3}^{\alpha}d\omega &\leq \int_{[-2\rho,2\rho]\times \{ \moc{\cos \varphi}\leq 3\rho \}} \moc{\cos \varphi}^{\alpha} d\varphi d\theta
                                                                    \lesssim \rho \int_{-3\rho}^{3\rho}\dfrac{\moc{x}^{\alpha}}{\sqrt{1-\moc{x}^2}}dx
         \lesssim \rho^{2+\alpha}.
        \end{align}
Together, \eqref{formu:angular-estimate-2}, \eqref{formu:angular-estimate-3}, and \eqref{formu:angular-estimate-4} establish \eqref{Estimate for Angular integration}.
    \end{proof}

We next compare the characteristic weights before and after collision.
    \begin{lemma}\label{lem:Weighted -Estimate}
Let $u$ and $v$ be the outgoing momenta after the collision of $p$ and $q$, ordered so that $u$ has the lower energy and $v$ the higher energy, equivalently $\moc{u}\leq \moc{v}$. With the notation $S_a$ introduced above, we have
       \begin{equation}\label{formu:Weighted -Estimate}
         \dfrac{S_p}{S_{u}S_{v}} \lesssim S^{-1}_{u}+ u^0\sqrt{\dfrac{v^0}{\mathfrak{s}p^0}}\left( S^{-1}_{u}+S^{-1}_{v} \right)\lesssim u^0\left( S^{-1}_{u}+S^{-1}_{v} \right) .
       \end{equation}
    \end{lemma}

    \begin{proof}
The outgoing velocity identity \eqref{formu:estimate for p'q'} yields
      \begin{align}\label{formu:Weighted -Estimate-1}
       \moc{\widehat{u} - \widehat{v}}&=\moc{\widehat {p}'-\widehat{q}'}
       =\moc{ \dfrac{g}{2p'^0q'^0}\left( EI - \dfrac{P\otimes P}{E+\sqrt{\mathfrak{s}}} \right)\omega}                    \geq \dfrac{g\sqrt{\mathfrak{s}}}{2u^0v^0},
      \end{align}
Here the positive matrix $ EI - \dfrac{P\otimes P}{E+\sqrt{\mathfrak{s}}} $ has eigenvalues $E,E,\sqrt{\mathfrak{s}}$. Energy and momentum conservation imply $g^2(p,v) \leq g^2$, while $p^0v^0\moc{\widehat{p}-\widehat{v}}^2\leq g^2$ gives
      \begin{equation}\label{formu:Weighted -Estimate-2}
        \moc{\widehat{p}-\widehat{v}}\leq  \dfrac{g}{\sqrt{p^0v^0}}.
      \end{equation}
Combining \eqref{formu:Weighted -Estimate-1} and
\eqref{formu:Weighted -Estimate-2} with the triangle inequality, we obtain
      \begin{align*}
      \notag  S_{p}&=\langle x- (1+t)\widehat{p} \rangle\lesssim  \langle x- (1+t)\widehat{v} \rangle  + (1+t)\moc{\widehat{p}-\widehat{v}}\\
      \notag                      &\lesssim \langle x- (1+t)\widehat{v} \rangle + u^0 \sqrt{\dfrac{v^0}{\mathfrak{s}p^0}} (1+t)\moc{\widehat{u}-\widehat{v}}\\
                            &\lesssim \langle x- (1+t)\widehat{v} \rangle + u^0\sqrt{\dfrac{v^0}{\mathfrak{s}p^0}}(\langle x- (1+t)\widehat{u} \rangle+ \langle x- (1+t)\widehat{v} \rangle )\\
       \notag                     &\lesssim S_v+ u^0\sqrt{\dfrac{v^0}{\mathfrak{s}p^0}}(S_u + S_v)
       \notag                      \lesssim u^0(S_u + S_v),
      \end{align*}
the final inequality uses $\sqrt{\dfrac{v^0}{\mathfrak{s}p^0}} \lesssim 1$, which follows from the definitions. Dividing both sides by the weights, we obtain
\eqref{formu:Weighted -Estimate}. This completes the proof.
    \end{proof}

Specially, Lemma~\ref{lem:Weighted -Estimate} yields
    \begin{equation}\label{Weighted -Estimate}
      \dfrac{W_p}{W_{p'}W_{q'}} = \left(\dfrac{p^0}{p'^0q'^0}\right)^{m}\left(\dfrac{S_{p}}{S_{p'}S_{q'}}\right)^{\ell}\lesssim  (u^0)^{\ell}\left(\dfrac{p^0}{p'^0q'^0}\right)^{m}\left( S_{u}^{-\ell} + S_{v}^{-\ell}\right).
    \end{equation}

The following lemma shows how bounds on $u^0$ or $S_u$
restrict $\omega$ to spherical caps.
    \begin{lemma}\label{Set subset}
Let $h,\rho>0$ and $g>0$. When $P\neq0$,
          \begin{equation}\label{formu:Set subset-1}
         \{\omega\in \mathbb{S}^2 | u^0\leq h \} \subset\mathcal{M}\left(\pm \frac{P}{\moc{P}},2\sqrt{\dfrac{h}{E}}\right).
       \end{equation}
If $\{ \omega\in \mathbb{S}^2| \moc{x-(1+t)\widehat{u}}\leq \rho \}$ is nonempty, choose $\omega_0 $ in this set. Then
       \begin{equation}\label{formu:Set subset}
       \begin{aligned}
           \{ \omega\in &\mathbb{S}^2|u^0\leq h \}\cap \{  \omega\in \mathbb{S}^2| \moc{x-(1+t)\widehat{u}}\leq \rho \}
           &\subset \mathcal{M}\left( \pm \omega_0, \dfrac{ChE\rho}{(1+t)g\sqrt{\mathfrak{s}}} \right),
       \end{aligned}
          \end{equation}
If $\{\omega:|x-(1+t)\widehat v(\omega)|\leq\rho\}$ is nonempty,
choose a point $\omega_v$ in that set. Then
        \begin{equation}\label{formu:Set subset-3}
        \begin{aligned}
              \{  \omega\in &\mathbb{S}^2| \moc{x-(1+t)\widehat{v}}\leq \rho \}
              &\subset \mathcal{M}\left( \pm \omega_v, \dfrac{CE^2\rho}{(1+t)g\sqrt{\mathfrak{s}}} \right).
        \end{aligned}
        \end{equation}
    \end{lemma}
    \begin{proof}
We first prove \eqref{formu:Set subset-1}.
By \eqref{defi:p'} and the definition of $u$,
$$
u^0=\dfrac{E}{2}\left(
1-\dfrac{g\moc{P}}{E\sqrt{\mathfrak s}}
\moc{\dfrac P{|P|}\cdot\omega}
\right).
$$
Thus, for $\omega\in\{\omega\in\mathbb{S}^2:u^0\leq h\}$,
we have
$$
\dfrac{g\moc{P}}{E\sqrt{\mathfrak{s}}}
\moc{\dfrac P{|P|}\cdot\omega}
\geq 1-\dfrac{2h}{E}.
$$
Since $\dfrac{g\moc{P}}{E\sqrt{\mathfrak{s}}}\leq1$, it follows that
$1-\dfrac{|P\cdot\omega|}{\moc{P}}\leq\dfrac{2h}{E}$.
This proves \eqref{formu:Set subset-1}.

Next, let
$\omega\in\{\omega\in\mathbb{S}^2:
\moc{x-(1+t)\widehat{u}}\leq\rho\}$.
Since $\omega_0$ belongs to the same set,
$$
\moc{\widehat{u}(\omega)-\widehat{u}(\omega_0)}
\leq
\dfrac{
\moc{x-(1+t)\widehat{u}(\omega)}
+\moc{x-(1+t)\widehat{u}(\omega_0)}
}{1+t}
\lesssim\dfrac{\rho}{1+t}.
$$
The choice of $u$ gives $u(-\omega)=u(\omega)$ almost everywhere.
Each of $u(\omega)$ and $u(\omega_0)$ is obtained from $p'$
by choosing the corresponding direction or its negative.
Applying \eqref{formu:Lip-estimate for omega} with these choices,
we obtain
$$
\min\{\moc{\omega+\omega_0},\moc{\omega-\omega_0}\}
\lesssim
\dfrac{u^0E}{g\sqrt{\mathfrak{s}}}
\moc{\widehat{u}(\omega)-\widehat{u}(\omega_0)}
\lesssim
\dfrac{u^0E\rho}{(1+t)g\sqrt{\mathfrak{s}}}.
$$
Consequently,
\begin{equation*}
\{\omega\in\mathbb{S}^2:
\moc{x-(1+t)\widehat{u}}\leq\rho\}
\subset
\left\{
\omega\in\mathbb{S}^2:
\min\{\moc{\omega+\omega_0},\moc{\omega-\omega_0}\}
\lesssim
\dfrac{u^0E\rho}{(1+t)g\sqrt{\mathfrak{s}}}
\right\}.
\end{equation*}
Combining this inclusion with $u^0\leq h$ gives
\eqref{formu:Set subset}.
For \eqref{formu:Set subset-3}, choose $\omega_v$ from the $v$-sublevel
set itself. Then $|\widehat v(\omega)-\widehat v(\omega_v)|\leq2\rho/(1+t)$.
Represent both outgoing momenta by the appropriate signs of their collision
directions and apply \eqref{formu:Lip-estimate for omega}, using $v^0\leq E$.
This gives the stated pair of caps. If a sublevel set is empty, its angular
integral is zero, so no center is needed.
This completes the proof of Lemma \ref{Set subset}.
\end{proof}

    \begin{remark}
        We do not treat the case $P=0$ separately, since for each fixed $p$, the set ${q=-p}$ has measure zero and therefore does not affect the  integrations. Actually, when $P=0$, we have $u^0=v^0=E/2$, and the above conclusion follows by direct verification. We omit further details.
     \end{remark}

We next apply the two cap estimates to the weighted angular integral.
Interpolating the resulting bounds gives the time decay required
for the nonlinear estimates.
    \begin{proposition}
    \label{key estimate-1}
Assume that $n > 3 - r'$. 
For every $e\in \mathbb{S}^2,$ one has
      \begin{equation}\label{key-estimate-u}
        \int_{\mathbb{S}^2} (u^0)^{-n}S_u^{-\ell r'}\dfrac{d\omega}{\moc{e\cdot \omega}^{\frac{r'}{r}}} \lesssim (1+t)^{-(1+\varepsilon)}\dfrac{E^{\frac{r'+3\varepsilon}{2}}}{(g\sqrt{\mathfrak{s}})^{1+\varepsilon}}
        (u_{\min}^0)^{-n+\frac{4-r'+\varepsilon}{2}},
      \end{equation}
      and
      \begin{equation}\label{key-estimate-v}
        \int_{\mathbb{S}^2} (u^0)^{-n}S_v^{-\ell r'}\dfrac{d\omega}{\moc{e\cdot \omega}^{\frac{r'}{r}}} \lesssim (1+t)^{-(1+\varepsilon)}\dfrac{E^{\frac{2+ r'+5\varepsilon}{2}}}{(g\sqrt{\mathfrak{s}})^{1+\varepsilon}}(u_{\min}^0)^{-n+\frac{2-r'-\varepsilon}{2}},
      \end{equation}
      where $u_{\min}^0 = \inf_{\omega\in \mathbb{S}^2} u^0.$
    \end{proposition}

    \begin{proof}
For $n>0$ and $\ell r'>0$, we have
\begin{equation}\label{layer-cake representations}
 (u^0)^{-n}
=n\int_{u^0}^{\infty}\xi^{-n-1}d\xi,
\qquad
S_u^{-\ell r'}
=\ell r'\int_{|x-(1+t)\widehat u|}^{\infty}
\frac{\rho}{(1+\rho^2)^{\ell r'/2+1}}d\rho.
\end{equation}
By Tonelli's theorem and Lemmas~\ref{lem:Local angular integration lemma}
and~\ref{Set subset}, we obtain
    \begin{align}\label{formu:I-gain-1}
   \notag 	\int_{\mathbb{S}^2}(u^0)^{-n}
    	& S_{u}^{-\ell r'}\dfrac{d\omega}{\moc{e\cdot \omega}^{\frac{r'}{r}}}\\
    \notag	&\lesssim\int_{u_{\min}^0}^{\infty}\dfrac{1}{\xi^{ n+1}}d\xi \int_{0}^{\infty}\dfrac{\rho}{(1+\rho^2)^{\ell r'/2 + 1}}d\rho\int_{\substack{\{\omega\in\mathbb S^2:\ u^0\leq\xi,\\
|x-(1+t)\widehat u|\leq\rho\}}} \dfrac{d\omega}{\moc{e\cdot \omega}^{\frac{r'}{r}}}\\
     \	&\lesssim\int_{u_{\min}^0}^{\infty}\dfrac{1}{\xi^{ n+1}}d\xi \int_{0}^{\infty}\dfrac{\rho}{(1+\rho^2)^{\ell r'/2 + 1}}d\rho \int_{\mathcal{M}(\pm \omega_0, \frac{C\xi E\rho}{(1+t)g\sqrt{\mathfrak{s}}})}\dfrac{d\omega}{\moc{e\cdot \omega}^{\frac{r'}{r}}}\\
     \notag	&\lesssim\dfrac{ E^{3-r'}}{(1+t)^{3-r'}g^{3-r'}\mathfrak{s}^{\frac{3-r'}{2}}}\int_{u_{\min}^0}^{\infty} \int_{0}^{\infty}
    	\xi^{-n+2-r'}
    	\dfrac{\rho^{4-r'}}{(1+\rho^2)^{\ell r'/2+1 }}d\rho d\xi\\
    	&\lesssim
    \notag	\dfrac{E^{3-r'}}{(1+t)^{3-r'}(g\sqrt{\mathfrak{s} })^{3-r'}} (u_{\min}^0)^{-n+3-r'}.
    \end{align}
For a second bound, we use $S_u^{-\ell r'}\leq1$ and
\eqref{formu:Set subset-1} to obtain
\begin{align}\label{formu:I-gain-2}
    \notag	\int_{\mathbb{S}^2}(u^0)^{-n}
    	& S_{u}^{-\ell r'}\dfrac{d\omega}{\moc{e\cdot \omega}^{\frac{r'}{r}}}
      \lesssim\int_{u_{\min}^0}^\infty \xi^{-n-1}\,d\xi
       \int_{\{\omega:u^0\leq\xi\}}\frac{d\omega}{|e\cdot\omega|^{r'/r}}\\
        & \lesssim \int_{u_{\min}^0}^{\infty} \xi^{-n-1} d\xi\int_{\mathcal{M}(\pm \frac{P}{\moc{P}}, 2\sqrt{\frac{\xi}{E}})} \dfrac{d\omega}{\moc{e\cdot \omega}^{\frac{r'}{r}}}\\
     \notag   &\lesssim E^{-\frac{3-r'}{2}} \int_{u_{\min}^0}^{\infty} \xi^{-n-1+\frac{3-r'}{2}}d\xi\lesssim E^{-\frac{3-r'}{2}}(u_{\min}^0)^{-n+\frac{3-r'}{2}}.
\end{align}
Interpolating \eqref{formu:I-gain-1} and \eqref{formu:I-gain-2}
with the exponents $\dfrac{1+\varepsilon}{3-r'}$ and
$\dfrac{2-\varepsilon-r'}{3-r'}$, respectively, we obtain
\begin{align*}
      \int_{\mathbb{S}^2}(u^0)^{-n}&
    	 S_{u}^{-\ell r'}\moc{e\cdot \omega}^{-\frac{r'}{r}}d\omega\\
    &\lesssim \left( \dfrac{E^{3-r'}}{(1+t)^{3-r'}(g\sqrt{\mathfrak{s} })^{3-r'}} (u_{\min}^0)^{-n+3-r'} \right)^{\frac{1+\varepsilon}{3-r'}}
    \left( E^{-\frac{3-r'}{2}}(u_{\min}^0)^{-n+\frac{3-r'}{2}} \right)^{\frac{2-\varepsilon-r'}{3-r'}}\\
    &\lesssim (1+t)^{-(1+\varepsilon)}
    \dfrac{E^{\frac{r'+3\varepsilon}{2}}}
    {(g\sqrt{\mathfrak{s}})^{1+\varepsilon}}
    (u_{\min}^0)^{-n+\frac{4-r'+\varepsilon}{2}},
\end{align*}
this proves \eqref{key-estimate-u}.
For the integral containing $S_v$, we use
\eqref{formu:Set subset-3} to obtain
\begin{align}\label{formu:I-gain-3}
    \notag	\int_{\mathbb{S}^2}(u^0)^{-n}
    	& S_{v}^{-\ell r'}\dfrac{d\omega}{\moc{e\cdot \omega}^{\frac{r'}{r}}}
    	  \lesssim (u_{\min}^0)^{-n}\int_{0}^{\infty} d\rho \int_{\moc{x-(1+t)\widehat{v}}\leq \rho} \dfrac{\rho}{(1+\rho^2)^{\ell r'/2 +1}} \dfrac{d\omega}{\moc{e\cdot \omega}^{\frac{r'}{r}}}\\
     	&\lesssim (u_{\min}^0)^{-n}\int_{0}^{\infty}\dfrac{\rho}{(1+\rho^2)^{\ell r'/2 + 1}}d\rho \int_{\mathcal{M}(\pm \omega_v,\frac{C\rho E^2}{(1+t)g\sqrt{\mathfrak{s}}})}\dfrac{d\omega}{\moc{e\cdot \omega}^{\frac{r'}{r}}}\\
   \notag 	&\lesssim\dfrac{(u_{\min}^0)^{-n} E^{6-2r'}}{(1+t)^{3-r'}g^{3-r'}\mathfrak{s}^{\frac{3-r'}{2}}} \int_{0}^{\infty}
    	\dfrac{\rho^{4-r'}}{(1+\rho^2)^{\ell r'/2+1}}d\rho \\
    	\notag&\lesssim
    	\dfrac{E^{6-2r'}}{(1+t)^{3-r'}(g\sqrt{\mathfrak{s} })^{3-r'}} (u_{\min}^0)^{-n}.
\end{align}
Applying the same argument as in \eqref{formu:I-gain-2},
we also have
\begin{align}\label{formu:I-gain-4}
    \int_{\mathbb{S}^2}(u^0)^{-n}
    	& S_{v}^{-\ell r'}\dfrac{d\omega}{\moc{e\cdot \omega}^{\frac{r'}{r}}}
    	 \lesssim\int_{u_{\min}^0}^{\infty} \dfrac{1}{\xi^{n+1}} d\xi \int_{u^0\leq \xi}\dfrac{d\omega}{\moc{e\cdot \omega}^{\frac{r'}{r}}}
    \lesssim E^{-\frac{3-r'}{2}}(u_{\min}^0)^{-n+\frac{3-r'}{2}}.
\end{align}
Interpolating \eqref{formu:I-gain-3} and \eqref{formu:I-gain-4}
with the same exponents proves \eqref{key-estimate-v}.
\end{proof}

    The preceding angular estimate is insufficient to treat the  range
$\gamma\in[0,1]$, so we use the following refined estimate for $g\geq1$.
    \begin{proposition}\label{key estimate-gamma>0}
      Assume that $g\geq1$, $2<r<\infty$, $\ell r'>3-r'$, and $n>6$. Then, for $t\geq0$,
      \begin{equation}\label{equ:key estimate-gamma>0}
        \int_{\mathbb S^2}(u^0)^{-n}S_u^{-\ell r'}\frac{d\omega}{|\widehat u_1-\widehat v_1|^{r'/r}}\lesssim \frac{(u_{\min}^0)^{-n+6}}{(1+t)^{3-r'}g^2}.
      \end{equation}
The implicit constant depends only on $r,n,\ell$ and is uniform in
$t,x,p,q$ on the stated region.
    \end{proposition}

   \begin{proof}
We use \eqref{layer-cake representations} as in the proof of
\eqref{formu:I-gain-1}, splitting the angular integral into
$u^0\leq \frac{E}{4}$ and $\frac{E}{4}<u^0\leq \frac{E}{2}$.

First suppose that $u^0\leq E/4$.
Fix $1\leq h\leq \frac{E}{4}$ and suppose that $\{u^0\leq h\}$
is nonempty. The identity
$u^0=\frac E2-\frac{g|P\cdot\omega|}{2\sqrt{\mathfrak s}}$,
together with $u^0\leq h$ and
$u_{\min}^0\geq \frac{E}{\mathfrak s}$, gives
\begin{equation*}
  \dfrac{\moc{P \cdot \omega}}{\moc{P}} \geq \dfrac{1}{2}, \quad \moc{\omega - \dfrac{P\cdot \omega}{\moc{P}^2}P}\leq \dfrac{2h}{g}, \quad E \leq h\mathfrak{s}.
\end{equation*}
From \eqref{defi:p'}, we have
$$
p'(\omega)-p'(\omega_0)
=
\dfrac{g}{2}
\left(
I+\dfrac{P\otimes P}
{\sqrt{\mathfrak{s}}(E+\sqrt{\mathfrak{s}})}
\right)(\omega-\omega_0).
$$
The differential of the map $\widehat u\mapsto u$ has norm at most
$(u^0)^3$. Applying these bounds with the signs corresponding to $u$,
we obtain, for
$\omega,\omega_0\in
\{u^0\leq h\}\cap
\{\moc{x-(1+t)\widehat u}\leq\rho\}$,
$$
\min\{\moc{\omega-\omega_0},\moc{\omega+\omega_0}\}
\leq\dfrac{2}{g}\moc{u(\omega)-u(\omega_0)}
\lesssim
\dfrac{h^3}{g}
\moc{\widehat u(\omega)-\widehat u(\omega_0)}
\lesssim\dfrac{h^3\rho}{(1+t)g}.
$$

Suppose first that $\moc{L}\geq \frac{E}{100h^2}$.
With $e=\dfrac{L}{\moc{L}}$, we have
$\moc{\widehat u_1-\widehat v_1}^{-\frac{r'}r}
\lesssim
\dfrac{h^{\frac{3r'}r}}
{g^{\frac{r'}r}\moc{e\cdot\omega}^{\frac{r'}r}}$.
Applying Lemma~\ref{lem:Local angular integration lemma}, we obtain
\begin{align}\label{gamma>0-2}
  \int_{\{u^0\leq h\}\cap \{\moc{x-(1+t)\widehat{u}}\leq \rho\}}\dfrac{d\omega}{\moc{\widehat{u}_1 - \widehat{v}_1}^{\frac{r'}{r}}}
  &\lesssim
  \dfrac{h^6}{g^2}
  \left(\dfrac{\rho}{1+t}\right)^{2-\frac{r'}r}.
\end{align}

If instead $\moc{L}\leq \frac{E}{100h^2}$, then
$\moc{L}\geq E\sqrt{1-\frac{P_1^2}{\moc{P}^2}}$
implies $\frac{|P_1|}{\moc{P}}\geq \dfrac{1}{2}$.
Writing
$L=E\left(e_1-\dfrac{P_1P}{\moc{P}^2}\right)
+\sqrt{\mathfrak{s}}\dfrac{P_1P}{\moc{P}^2}$,
we obtain
\begin{align*}
  \moc{L\cdot \omega}
  &\geq
  \sqrt{\mathfrak{s}}
  \dfrac{\moc{P_1}\moc{P\cdot \omega}}{\moc{P}^2}
  -E\sqrt{1-\dfrac{P_1^2}{\moc{P}^2}}
  \moc{\omega-\dfrac{P\cdot \omega}{\moc{P}^2}P}\\
  &\geq
  \dfrac{\sqrt{\mathfrak{s}}}{4}
  -\dfrac{2Eh}{100h^2g}\\
  &\geq
  \dfrac{\sqrt{\mathfrak{s}}}{4}
  -\dfrac{\mathfrak{s}}{50g}
  \geq\dfrac{\sqrt{\mathfrak{s}}}{10},
\end{align*}
where we used $\mathfrak{s}\leq5g^2$, which follows from $g\geq1$.
Together with $E\leq h\mathfrak{s}$, this gives
$$
\dfrac{1}
{\moc{\widehat u_1-\widehat v_1}^{\frac{r'}r}}
\lesssim
\left(\dfrac{hE}{g\sqrt{\mathfrak{s}}}\right)^{\frac{r'}r}
\lesssim h^{\frac{2r'}r}.
$$
Applying Lemma~\ref{lem:Local angular integration lemma}
with $\alpha=0$, we obtain
\begin{equation}\label{gamma>0-3}
  \begin{aligned}
  \int_{\{u^0\leq h\}\cap \{\moc{x-(1+t)\widehat{u}}\leq \rho\}}\dfrac{d\omega}{\moc{\widehat{u}_1 - \widehat{v}_1}^{\frac{r'}{r}}}
  &\lesssim
  h^{\frac{2r'}r}
  \int_{\{u^0\leq h\}\cap
  \{\moc{x-(1+t)\widehat u}\leq\rho\}}d\omega\\
  &\lesssim
  h^{\frac{2r'}r}
  \min\left\{
  \dfrac{h}{g},
  \dfrac{h^3\rho}{(1+t)g}
  \right\}^2
  \lesssim
  \dfrac{h^6}{g^2}
  \left(\dfrac{\rho}{1+t}\right)^{3-r'}.
\end{aligned}
\end{equation}
Here we also used
$\min\{\moc{\omega-\omega_0},\moc{\omega+\omega_0}\}
\lesssim \frac{h}{g}$.

Combining \eqref{gamma>0-2} and \eqref{gamma>0-3},
and taking $h=\min\{\xi,E/4\}$ in
\eqref{layer-cake representations}, we obtain
\begingroup\small
\begin{align*}
  \int_{u^0\leq E/4}(u^0)^{-n}S_u^{-\ell r'}
  &\dfrac{d\omega}
  {\moc{\widehat u_1-\widehat v_1}^{\frac{r'}r}}\\
  &\lesssim
  \int_{u_{\min}^0}^{\infty}\xi^{-(n+1)}d\xi
  \int_0^\infty
  \dfrac{\rho\,d\rho}
  {(1+\rho^2)^{\frac{\ell r'}2+1}}
  \cdot
  \int_{\{u^0\leq\min\{\xi,E/4\},
  \moc{x-(1+t)\widehat u}\leq\rho\}}
  \dfrac{d\omega}
  {\moc{\widehat u_1-\widehat v_1}^{\frac{r'}r}}\\
  &\lesssim
  \dfrac{1}{(1+t)^{3-r'}g^2}
  \int_{u_{\min}^0}^{\infty}
  \xi^{-n-1}\left[\min\left\{\xi,\dfrac E4\right\}\right]^6d\xi
  \int_0^\infty
  \dfrac{\rho^{4-r'}}
  {(1+\rho^2)^{\frac{\ell r'}2+1}}d\rho\\
  &\lesssim
  \dfrac{(u_{\min}^0)^{-n+6}}
  {(1+t)^{3-r'}g^2}.
\end{align*}
\endgroup

It remains to consider $E/4\leq u^0\leq E/2$.
Arguing as in \eqref{formu:I-gain-3}, we obtain
\begin{align*}
\notag
\int_{E/4\leq u^0\leq E/2}
(u^0)^{-n}S_u^{-\ell r'}
\dfrac{d\omega}
{\moc{\widehat u_1-\widehat v_1}^{\frac{r'}r}}
&\lesssim
E^{-n}
\dfrac{E^{\frac{2r'}r}}
{(g\sqrt{\mathfrak{s}})^{\frac{r'}r}}
\int_{\mathbb S^2}
S_u^{-\ell r'}
\dfrac{d\omega}{\moc{e\cdot\omega}^{\frac{r'}r}}\\
&\lesssim
E^{-n}
\dfrac{E^{\frac{2r'}r}}
{(g\sqrt{\mathfrak{s}})^{\frac{r'}r}}
\dfrac{E^{2(3-r')}}
{(1+t)^{3-r'}(g\sqrt{\mathfrak{s}})^{3-r'}}\\
\notag
&\lesssim
\dfrac{(u_{\min}^0)^{-n+4}}
{(1+t)^{3-r'}g^2}.
\end{align*}

Since $u_{\min}^0\geq1$, the two bounds give
\eqref{equ:key estimate-gamma>0}.
This completes the proof.
\end{proof}
It remains to combine the angular bounds with integration over the partner momentum. We retain the energy and relative-momentum decompositions explicitly, since they compensate for the velocity degeneracy in the angular estimates.

\begin{proposition}[Decay of the gain factor]
\label{decay of I-gain}
Under the parameter assumptions \eqref{equ:gamma-r-1},
\eqref{equ:gamma-r-2}, and \eqref{equ:vare}, the gain factor
in \eqref{I-gain} satisfies, uniformly for $t\geq0$ and $x,p\in\mathbb R^3$,
\[
 I_{\mathrm{gain}}^{r'}(t,x,p)
 \lesssim(1+t)^{-(1+\varepsilon)}.
\]
\end{proposition}

\begin{proof}
We first consider $\gamma\in(-2,0)$.

Let $L=Ee_1-
\frac{P_1P}{E+\sqrt{\mathfrak s}}$ and $e=\frac{L}{|L|}$.
The direction $e$ depends only on $p,q$.
By \eqref{formu:estimate for p'q'}, we have
$$
\moc{\widehat p'_1-\widehat q'_1}
=
\dfrac{g\moc{L\cdot\omega}}{2p'^0q'^0},
$$
where $|L|^2=E^2-P_1^2\geq\mathfrak s$. Hence
$$
\frac{1}
{\moc{\widehat p'_1-\widehat q'_1}^{\frac{r'}r}}
\thicksim
\dfrac{(u^0v^0)^{\frac{r'}r}}
{g^{\frac{r'}r}\moc{L\cdot\omega}^{\frac{r'}r}}
\thicksim
\dfrac{(u^0v^0)^{\frac{r'}r}}
{g^{\frac{r'}r}\moc{L}^{\frac{r'}r}
\moc{e\cdot\omega}^{\frac{r'}r}}.
$$
Using \eqref{Weighted -Estimate},
$p^0\leq E\leq2v^0$, and
$\moc{L}\geq\mathfrak s^{1/2}$, we obtain
\begin{align}\label{I-gain-es-2}
\notag
I_{\mathrm{gain}}^{r'}
&\lesssim
\int_{\mathbb R^3\times\mathbb S^2}
\dfrac{
g^{\gamma r'}\mathfrak s^{\frac{r'}2}
W_p^{r'}(p^0q^0)^{\frac{r'}r}}
{(p^0q^0)^{r'}W_{p'}^{r'}W_{q'}^{r'}
\moc{\widehat p'_1-\widehat q'_1}^{\frac{r'}r}}
\,dq\,d\omega
\\
\notag
&\lesssim
\int_{\mathbb R^3\times\mathbb S^2}
\left(\dfrac{p^0}{u^0v^0}\right)^{mr'}
\left(\dfrac{S_p}{S_uS_v}\right)^{\ell r'}
\dfrac{
g^{\gamma r'}\mathfrak s^{\frac{r'}2}
(p^0q^0)^{\frac{r'}r}}
{(p^0q^0)^{r'}
\moc{\widehat p'_1-\widehat q'_1}^{\frac{r'}r}}
\,dq\,d\omega
\\
&\lesssim
\int_{\mathbb R^3\times\mathbb S^2}
\left(\dfrac{p^0}{u^0v^0}\right)^{mr'}
\left(\dfrac{S_p}{S_uS_v}\right)^{\ell r'}
\dfrac{
g^{\gamma r'}\mathfrak s^{\frac{r'}2}
(p^0q^0)^{\frac{r'}r}}
{(p^0q^0)^{r'}}
\dfrac{(u^0v^0)^{\frac{r'}r}}
{g^{\frac{r'}r}\moc{L}^{\frac{r'}r}
\moc{e\cdot\omega}^{\frac{r'}r}}
\,dq\,d\omega
\\
\notag
&\lesssim
\int_{\mathbb R^3\times\mathbb S^2}
\left(\dfrac{p^0}{u^0v^0}\right)^{mr'}
\left(\dfrac{S_p}{S_uS_v}\right)^{\ell r'}
\dfrac{
g^{r'(\gamma-\frac1r)}\mathfrak s^{\frac12}}
{p^0q^0}
\dfrac{(u^0v^0)^{\frac{r'}r}}
{\moc{e\cdot\omega}^{\frac{r'}r}}
\,dq\,d\omega
\\
\notag
&\lesssim
(p^0)^{mr'}
\int_{\mathbb R^3}
\dfrac{g^{(\gamma-1)r'+1}\sqrt{\mathfrak s}}
{p^0q^0}\,dq
\int_{\mathbb S^2}
\left(\dfrac{S_p}{S_uS_v}\right)^{\ell r'}
\dfrac{d\omega}
{(u^0v^0)^{(m-1)r'+1}
\moc{e\cdot\omega}^{\frac{r'}r}}.
\end{align}

Consider first $p^0\leq2$.
The argument of Proposition~\ref{key estimate-1} gives
$$
\int_{\mathbb S^2}
(S_u^{-\ell r'}+S_v^{-\ell r'})
\dfrac{d\omega}{\moc{e\cdot\omega}^{\frac{r'}r}}
\lesssim
\min\left\{
1,
\left(
\dfrac{E^2}{(1+t)g\sqrt{\mathfrak s}}
\right)^{3-r'}
\right\}.
$$
Since $1+\varepsilon<3-r'$, this implies
\begin{equation}\label{I-gain-es-ex1}
\int_{\mathbb S^2}
(S_u^{-\ell r'}+S_v^{-\ell r'})
\dfrac{d\omega}{\moc{e\cdot\omega}^{\frac{r'}r}}
\lesssim
\left(
\dfrac{E^2}{(1+t)g\sqrt{\mathfrak s}}
\right)^{1+\varepsilon}.
\end{equation}
Combining \eqref{Weighted -Estimate},
\eqref{I-gain-es-ex1}, and $E/2\leq v^0\leq E$, we obtain
\begin{align*}
\int_{\mathbb S^2}
(u^0v^0)^{-[(m-1)r'+1]}
&\left(\dfrac{S_p}{S_uS_v}\right)^{\ell r'}
\dfrac{d\omega}{\moc{e\cdot\omega}^{\frac{r'}r}}
\\
&\lesssim
\int_{\mathbb S^2}
(u^0)^{\ell r'-(m-1)r'-1}
E^{-(m-1)r'-1}
(S_u^{-\ell r'}+S_v^{-\ell r'})
\dfrac{d\omega}{\moc{e\cdot\omega}^{\frac{r'}r}}
\\
&\lesssim
E^{-(m-1)r'-1}
\left(
\dfrac{E^2}{(1+t)g\sqrt{\mathfrak s}}
\right)^{1+\varepsilon}.
\end{align*}
Substituting this bound into \eqref{I-gain-es-2}
and using $p^0\leq2$, we obtain
\begin{align}\label{I-gain-es-1}
\notag
I_{\mathrm{gain}}^{r'}
&\lesssim
(1+t)^{-(1+\varepsilon)}
\int_{\mathbb R^3}
E^{-(m-1)r'+1+2\varepsilon}
\dfrac{g^{(\gamma-1)r'-\varepsilon}}{q^0}\,dq
\\
&\lesssim
(1+t)^{-(1+\varepsilon)}
\left(
\int_{q^0\leq20}+\int_{q^0\geq20}
\right)
E^{-(m-1)r'+1+2\varepsilon}
\dfrac{g^{(\gamma-1)r'-\varepsilon}}{q^0}\,dq
\\
\notag
&:=
(1+t)^{-(1+\varepsilon)}(J_1+J_2).
\end{align}
For $J_1$, \eqref{formu:id-g} gives
$\moc{p-q}\lesssim g$, and hence
$$
J_1
\lesssim
\int_{q^0\leq20}g^{(\gamma-1)r'-\varepsilon}\,dq
\lesssim
\int_{\moc{q-p}\leq100}
\moc{p-q}^{(\gamma-1)r'-\varepsilon}\,dq
\lesssim1.
$$
For $J_2$, we use $q^0\thicksim E$ and $q^0\lesssim g^2$
to obtain
$$
J_2
\lesssim
\int_{q^0\geq20}
\dfrac{
E^{-(m-1)r'+1+2\varepsilon}
E^{\frac{(\gamma-1)r'-\varepsilon}{2}}
}{q^0}\,dq
\lesssim
\int_{q^0\geq20}
\langle q\rangle^{
-(m-1)r'+2\varepsilon+
\frac{(\gamma-1)r'-\varepsilon}{2}}
\,dq
\lesssim1.
$$
Thus \eqref{I-gain-es-1} imples
$I_{\mathrm{gain}}^{r'}\lesssim(1+t)^{-(1+\varepsilon)}$
when $p^0\leq2$.

For $p^0\geq2$, we split \eqref{I-gain-es-2} at $g=1$:
\begin{align*}
I_{\mathrm{gain}}^{r'}
&\lesssim
(p^0)^{mr'}
\int_{\mathbb R^3}
\dfrac{g^{(\gamma-1)r'+1}\sqrt{\mathfrak s}}
{p^0q^0}\,dq
\int_{\mathbb S^2}
\left(\dfrac{S_p}{S_uS_v}\right)^{\ell r'}
\dfrac{d\omega}
{(u^0v^0)^{(m-1)r'+1}
\moc{e\cdot\omega}^{\frac{r'}r}}
\\
&=
(p^0)^{mr'}
\left(\int_{g\leq1}+\int_{g\geq1}\right)
\dfrac{g^{(\gamma-1)r'+1}\sqrt{\mathfrak s}}
{p^0q^0}\,dq
\int_{\mathbb S^2}
\left(\dfrac{S_p}{S_uS_v}\right)^{\ell r'}
\dfrac{d\omega}
{(u^0v^0)^{(m-1)r'+1}
\moc{e\cdot\omega}^{\frac{r'}r}}
\\
&:=J_3+J_4.
\end{align*}
On the region $g\leq1$, \eqref{formu:id-g} implies
$$
\moc{p^0-q^0}^2
\leq\moc{p-q}^2
\leq g^2p^0q^0
\leq p^0q^0.
$$
Solving this inequality for $q^0/p^0$, we obtain
$$
1+\frac{g^2}{2}
-g\sqrt{1+\frac{g^2}{4}}
\leq\frac{q^0}{p^0}
\leq
1+\frac{g^2}{2}
+g\sqrt{1+\frac{g^2}{4}},
$$
thus $q^0\thicksim p^0\thicksim E$.
Since $g(u,v)=g(p,q)\leq1$, the same comparison gives
$u^0\thicksim v^0\thicksim q^0\thicksim p^0\thicksim E$.
By \eqref{Weighted -Estimate} and \eqref{I-gain-es-ex1},
\begin{align*}
\int_{\mathbb S^2}
\left(\dfrac{S_p}{S_uS_v}\right)^{\ell r'}
&\dfrac{d\omega}
{(u^0v^0)^{(m-1)r'+1}
\moc{e\cdot\omega}^{\frac{r'}r}}
\\
&\lesssim
(p^0)^{-2(m-1)r'-2}
\int_{\mathbb S^2}
(u^0)^{\ell r'}
(S_u^{-\ell r'}+S_v^{-\ell r'})
\dfrac{d\omega}{\moc{e\cdot\omega}^{\frac{r'}r}}
\\
&\lesssim
(p^0)^{-2(m-1)r'-2+\ell r'}
\int_{\mathbb S^2}
(S_u^{-\ell r'}+S_v^{-\ell r'})
\dfrac{d\omega}{\moc{e\cdot\omega}^{\frac{r'}r}}
\\
&\lesssim
(p^0)^{-2(m-1)r'-2+\ell r'}
\dfrac{E^{2(1+\varepsilon)}}
{(1+t)^{1+\varepsilon}
g^{1+\varepsilon}
\mathfrak s^{\frac{1+\varepsilon}{2}}}
\\
&\lesssim
(p^0)^{-2(m-1)r'-2+\ell r'+2(1+\varepsilon)}
\dfrac{1}
{(1+t)^{1+\varepsilon}
g^{1+\varepsilon}
\mathfrak s^{\frac{1+\varepsilon}{2}}}.
\end{align*}
Substitution into $J_3$ gives
$$
J_3\lesssim
(1+t)^{-(1+\varepsilon)}
(p^0)^{-(m-2-\ell)r'+2\varepsilon-1}
\int_{g\leq1}
g^{(\gamma-1)r'-\varepsilon}\frac{dq}{q^0}.
$$
To estimate the remaining integral, fix $p\neq0$ and use
spherical coordinates with polar axis $p/|p|$.
Let $\theta\in[0,\pi]$ be the angle between $p$ and $q$,
and let $\varphi\in[0,2\pi)$ be the azimuthal angle. Then
$$
dq=|q|^2\sin\theta\,d|q|\,d\theta\,d\varphi.
$$
The relations
$$g^2=2(p^0q^0-\moc p|q|\cos\theta-1),$$
and
$$g\,dg=p^0\,dq^0-\moc p\cos\theta\,d\moc{q}
+\moc p|q|\sin\theta\,d\theta,\ dq^0=|q|\,d|q|/q^0,$$ give
\begin{equation}\label{Substituting-polar}
dq=\dfrac{gq^0}{\moc p}\,dq^0\,dg\,d\varphi.
\end{equation}
Since $q^0\thicksim p^0\thicksim|p|$ on this region,
the bounds on $q^0/p^0$ yield
\begin{align*}
(1+t)^{1+\varepsilon}J_3
&\lesssim
(p^0)^{-(m-2-\ell)r'+2\varepsilon-1}
\int_{g\leq1}
\dfrac{g^{(\gamma-1)r'-\varepsilon}}{q^0}\,dq
\\
&\lesssim
(p^0)^{-(m-2-\ell)r'+2\varepsilon-1}
\int_{g\leq1}
\int_{p^0(1+g^2/2-g\sqrt{1+g^2/4})}
^{p^0(1+g^2/2+g\sqrt{1+g^2/4})}
\dfrac{g^{(\gamma-1)r'-\varepsilon+1}}{q^0}
\,dq^0\,dg
\\
&\lesssim
(p^0)^{-(m-2-\ell)r'+2\varepsilon-1}
\int_0^1
g\sqrt{1+g^2/4}\,
g^{(\gamma-1)r'-\varepsilon+1}\,dg
\\
&\lesssim
\int_0^1g^{(\gamma-1)r'-\varepsilon+2}\,dg
\lesssim1.
\end{align*}
Therefore $J_3\lesssim(1+t)^{-(1+\varepsilon)}$.

For $J_4$, we have $g\geq1$ and
$\sqrt{\mathfrak s}\thicksim g$.
Using \eqref{formu:Weighted -Estimate} and
Proposition~\ref{key estimate-1}, we obtain
\begin{align*}
\int_{\mathbb S^2}
\left(\dfrac{S_p}{S_uS_v}\right)^{\ell r'}
&\dfrac{d\omega}
{(u^0v^0)^{(m-1)r'+1}
\moc{e\cdot\omega}^{\frac{r'}r}}
\\
&\lesssim
E^{-[(m-1)r'+1]}
\int_{\mathbb S^2}
(u^0)^{-[(m-1)r'+1]}S_u^{-\ell r'}
\dfrac{d\omega}{\moc{e\cdot\omega}^{\frac{r'}r}}
\\
&\quad+
\dfrac{E^{-[(m-1-\frac\ell2)r'+1]}}
{g^{\ell r'}(p^0)^{\frac{\ell r'}2}}
\int_{\mathbb S^2}
(u^0)^{-(m-1-\ell)r'-1}S_u^{-\ell r'}
\dfrac{d\omega}{\moc{e\cdot\omega}^{\frac{r'}r}}
\\
&\quad+
\dfrac{E^{-[(m-1-\frac\ell2)r'+1]}}
{g^{\ell r'}(p^0)^{\frac{\ell r'}2}}
\int_{\mathbb S^2}
(u^0)^{-(m-1-\ell)r'-1}S_v^{-\ell r'}
\dfrac{d\omega}{\moc{e\cdot\omega}^{\frac{r'}r}}
\\
&\lesssim
(1+t)^{-(1+\varepsilon)}
\Bigg(
(u_{\min}^0)^{-(m-\frac12)r'+1+\frac\varepsilon2}
\dfrac{
E^{-(m-\frac32)r'+\frac{3\varepsilon}{2}-1}}
{g^{2(1+\varepsilon)}}
\\
&\qquad+
(u_{\min}^0)^{-(m-\frac12-\ell)r'+1+\frac\varepsilon2}
\dfrac{
E^{-(m-\frac32-\frac\ell2)r'
+\frac{3\varepsilon}{2}-1}}
{g^{\ell r'+2(1+\varepsilon)}
(p^0)^{\frac{\ell r'}2}}
\\
&\qquad+
(u_{\min}^0)^{-(m-\frac12-\ell)r'-\frac\varepsilon2}
\dfrac{
E^{-(m-\frac32-\frac\ell2)r'
+\frac{5\varepsilon}{2}}}
{g^{\ell r'+2(1+\varepsilon)}
(p^0)^{\frac{\ell r'}2}}
\Bigg).
\end{align*}
Substituting into $J_4$, we obtain
\begin{align*}
(1+t)^{1+\varepsilon}J_4
&\lesssim
\int_{g\geq1}
\dfrac{
(p^0)^{mr'-1}
g^{(\gamma-1)r'-2\varepsilon}
E^{-(m-\frac32)r'+\frac{3\varepsilon}{2}-1}}
{q^0}
(u_{\min}^0)^{-(m-\frac12)r'+1+\frac\varepsilon2}
\,dq
\\
&\quad+
\int_{g\geq1}
\dfrac{
(p^0)^{mr'-1-\frac{\ell r'}2}
g^{(\gamma-1)r'-2\varepsilon-\ell r'}
E^{-(m-\frac32-\frac\ell2)r'
+\frac{3\varepsilon}{2}-1}}
{q^0}
(u_{\min}^0)^{-(m-\frac12-\ell)r'+1+\frac\varepsilon2}
\,dq
\\
&\quad+
\int_{g\geq1}
\dfrac{
(p^0)^{mr'-1-\frac{\ell r'}2}
g^{(\gamma-1)r'-2\varepsilon-\ell r'}
E^{-(m-\frac32-\frac\ell2)r'
+\frac{5\varepsilon}{2}}}
{q^0}
(u_{\min}^0)^{-(m-\frac12-\ell)r'-\frac\varepsilon2}
\,dq
\\
&:=J_{4,1}+J_{4,2}+J_{4,3}.
\end{align*}

Since $u_{\min}^0\leq E$, we have
\begin{align*}
E^{-(m-\frac32-\frac\ell2)r'
+\frac{3\varepsilon}{2}-1}
(u_{\min}^0)^{-(m-\frac12-\ell)r'+1+\frac\varepsilon2}
&=
\left(\frac{u_{\min}^0}{E}\right)^{1+\varepsilon}
E^{-(m-\frac32-\frac\ell2)r'
+\frac{5\varepsilon}{2}}
(u_{\min}^0)^{-(m-\frac12-\ell)r'-\frac\varepsilon2}
\\
&\leq
E^{-(m-\frac32-\frac\ell2)r'
+\frac{5\varepsilon}{2}}
(u_{\min}^0)^{-(m-\frac12-\ell)r'-\frac\varepsilon2}.
\end{align*}
Hence $J_{4,2}\leq J_{4,3}$, and it suffices to estimate
$J_{4,1}$ and $J_{4,3}$.

Using \eqref{Substituting-polar} and the bounds
$u_{\min}^0\geq E/\mathfrak s\thicksim E/g^2$ and
$u_{\min}^0\geq C\mathfrak s/E\thicksim g^2/E$,
we split the $g$-integration at $g=\sqrt E$:
\begin{align*}
J_{4,1}
&\lesssim
\int_{g\geq1}
(p^0)^{mr'-2}
g^{(\gamma-1)r'-2\varepsilon+1}
E^{-(m-\frac32)r'-1+\frac{3\varepsilon}{2}}
(u_{\min}^0)^{-(m-\frac12)r'+1+\frac\varepsilon2}
\,dq^0\,dg
\\
&\lesssim
\int_0^\infty
E^{-(m-\frac32)r'+\frac{3\varepsilon}{2}-1}
(p^0)^{mr'-2}\,dq^0
\\
&\quad\cdot
\Bigg[
\int_1^{\sqrt E}
g^{(\gamma-1)r'-2\varepsilon+1}
\left(\dfrac E{g^2}\right)^{
-(m-\frac12)r'+1+\frac\varepsilon2}
\,dg
\\
&\qquad+
\int_{\sqrt E}^\infty
g^{(\gamma-1)r'-2\varepsilon+1}
\left(\dfrac{g^2}E\right)^{
-(m-\frac12)r'+1+\frac\varepsilon2}
\,dg
\Bigg]
:=K_1+K_2.
\end{align*}
Integrating in $g$ and rescaling $q^0$ by $p^0$, we obtain
\begin{align}\label{K1}
\notag K_1
&\lesssim
\int_0^\infty
E^{-(m-\frac32)r'+\frac{3\varepsilon}{2}-1}
(p^0)^{mr'-2}\,dq^0
\int_1^{\sqrt E}
g^{(\gamma-1)r'-2\varepsilon+1}
\left(\dfrac E{g^2}\right)^{
-(m-\frac12)r'+1+\frac\varepsilon2}
\,dg
\\
&\lesssim
\int_0^\infty
(p^0)^{mr'-2}
E^{(-m+\frac\gamma2+1)r'+\frac\varepsilon2}
\,dq^0
\\
\notag&\lesssim
(p^0)^{mr'-2}p^0
(p^0)^{(-m+\frac\gamma2+1)r'+\frac\varepsilon2}
\int_0^\infty
(1+q^0)^{(-m+\frac\gamma2+1)r'+\frac\varepsilon2}
\,dq^0
\\
\notag&\lesssim
(p^0)^{(\frac\gamma2+1)r'-1+\frac\varepsilon2}
\lesssim1.
\end{align}
Similarly,
\begin{align*}
K_2
&\lesssim
\int_0^\infty
E^{-(m-\frac32)r'+\frac{3\varepsilon}{2}-1}
(p^0)^{mr'-2}\,dq^0
\int_{\sqrt E}^\infty
g^{(\gamma-1)r'-2\varepsilon+1}
\left(\dfrac{g^2}E\right)^{
-(m-\frac12)r'+1+\frac\varepsilon2}
\,dg
\\
&\lesssim
\int_0^\infty
(p^0)^{mr'-2}
E^{(-m+\frac\gamma2+1)r'+\frac\varepsilon2}
\,dq^0
\\
&\lesssim
(p^0)^{mr'-2}p^0
(p^0)^{(-m+\frac\gamma2+1)r'+\frac\varepsilon2}
\int_0^\infty
(1+q^0)^{(-m+\frac\gamma2+1)r'+\frac\varepsilon2}
\,dq^0
\\
&\lesssim
(p^0)^{(\frac\gamma2+1)r'-1+\frac\varepsilon2}
\lesssim1.
\end{align*}
Thus $J_{4,1}\lesssim1$.

For $J_{4,3}$, the same decomposition gives
\begin{align*}
J_{4,3}
&\lesssim
\int_{g\geq1}
(p^0)^{mr'-2-\frac{\ell r'}2}
g^{(\gamma-1)r'-2\varepsilon-\ell r'+1}
E^{-(m-\frac32-\frac\ell2)r'+\frac{5\varepsilon}{2}}
(u_{\min}^0)^{-(m-\frac12-\ell)r'-\frac\varepsilon2}
\,dq^0\,dg
\\
&\lesssim
\int_0^\infty
E^{-(m-\frac32-\frac\ell2)r'+\frac{5\varepsilon}{2}}
(p^0)^{mr'-2-\frac{\ell r'}2}\,dq^0
\\
&\quad\cdot
\Bigg[
\int_1^{\sqrt E}
g^{(\gamma-1)r'-2\varepsilon-\ell r'+1}
\left(\dfrac E{g^2}\right)^{
-(m-\frac12-\ell)r'-\frac\varepsilon2}
\,dg
\\
&\qquad+
\int_{\sqrt E}^\infty
g^{(\gamma-1)r'-2\varepsilon-\ell r'+1}
\left(\dfrac{g^2}E\right)^{
-(m-\frac12-\ell)r'-\frac\varepsilon2}
\,dg
\Bigg]
:=K_3+K_4.
\end{align*}
Integrating as in \eqref{K1}, we obtain
\begin{align*}
K_3
&\lesssim
\int_0^\infty
E^{-(m-\frac32-\frac\ell2)r'+\frac{5\varepsilon}{2}}
(p^0)^{mr'-2-\frac{\ell r'}2}\,dq^0
\int_1^{\sqrt E}
g^{(\gamma-1)r'-2\varepsilon-\ell r'+1}
\left(\dfrac E{g^2}\right)^{
-(m-\frac12-\ell)r'-\frac\varepsilon2}
\,dg
\\
&\lesssim
\int_0^\infty
(p^0)^{mr'-2-\frac{\ell r'}2}
E^{(-m+\frac\gamma2+1)r'+1+\frac{3\varepsilon}{2}}
\,dq^0
\\
&\lesssim
(p^0)^{mr'-2-\frac{\ell r'}2}p^0
(p^0)^{(-m+\frac\gamma2+1)r'+1+\frac{3\varepsilon}{2}}
\int_0^\infty
(1+q^0)^{(-m+\frac\gamma2+1)r'+1+\frac{3\varepsilon}{2}}
\,dq^0
\\
&\lesssim
(p^0)^{(\frac\gamma2+1)r'
+\frac{3\varepsilon}{2}-\frac{\ell r'}2}
\lesssim1.
\end{align*}
Similarly,
\begin{align*}
K_4
&\lesssim
\int_0^\infty
E^{-(m-\frac32-\frac\ell2)r'+\frac{5\varepsilon}{2}}
(p^0)^{mr'-2-\frac{\ell r'}2}\,dq^0
\int_{\sqrt E}^\infty
g^{(\gamma-1)r'-2\varepsilon-\ell r'+1}
\left(\dfrac{g^2}E\right)^{
-(m-\frac12-\ell)r'-\frac\varepsilon2}
\,dg
\\
&\lesssim
\int_0^\infty
(p^0)^{mr'-2-\frac{\ell r'}2}
E^{(-m+\frac\gamma2+1)r'+1+\frac{3\varepsilon}{2}}
\,dq^0
\\
&\lesssim
(p^0)^{mr'-2-\frac{\ell r'}2}p^0
(p^0)^{(-m+\frac\gamma2+1)r'+1+\frac{3\varepsilon}{2}}
\int_0^\infty
(1+q^0)^{(-m+\frac\gamma2+1)r'+1+\frac{3\varepsilon}{2}}
\,dq^0
\\
&\lesssim
(p^0)^{(\frac\gamma2+1)r'
+\frac{3\varepsilon}{2}-\frac{\ell r'}2}
\lesssim1.
\end{align*}
Together with $J_{4,2}\leq J_{4,3}$, these bounds give
$J_4\lesssim(1+t)^{-(1+\varepsilon)}$.
This completes the proof for $\gamma\in(-2,0)$.

We next consider $\gamma\in[0,1]$.

After taking out the common factor $(1+t)^{-(1+\varepsilon)}$,
we use the same notation $J_1,J_2,J_3,J_4$ for the remaining
integrals. As in the soft-potential case,
$$
\begin{aligned}
J_1
&\lesssim
\int_{q^0\leq20}
\moc{p-q}^{\gamma-\frac{r'}r-1-\varepsilon}\,dq
\lesssim
\int_{\moc{p-q}\leq100}
\moc{p-q}^{\gamma-\frac{r'}r-1-\varepsilon}\,dq
\lesssim1,
\\[2ex]
J_2
&\lesssim
\int_{q^0\geq20}
\langle q\rangle^{
-(m-1)r'+\frac{\gamma-1}{2}
+\varepsilon-\frac{r'}r}
\,dq
\lesssim1,
\\[2ex]
J_3
&\lesssim
(p^0)^{-(m-2-\ell)r'+2\varepsilon-1}
\int_0^1
g\sqrt{1+\frac{g^2}{4}}\,
g^{\gamma-\frac{r'}r-\varepsilon}\,dg
\lesssim1.
\end{aligned}
$$
Here $J_1$ and $J_2$ correspond to $p^0\leq2$, with
$q^0\leq20$ and $q^0\geq20$, respectively, while $J_3$
corresponds to $p^0\geq2$ and $g\leq1$.

It remains to estimate $J_4$, where $p^0\geq2$ and $g\geq1$.
We use the same decomposition
$J_4\lesssim J_{4,1}+J_{4,2}+J_{4,3}$ as in the soft-potential case.
For $J_{4,2}$ and $J_{4,3}$, we repeat the preceding estimates
with the corresponding powers of $g$ and $\sqrt{\mathfrak s}$.
This gives
$$
J_{4,2}\lesssim J_{4,3}
\lesssim
(p^0)^{
r'(\frac12-\frac\ell2)
+\frac{\gamma+1}{2}+\frac{3\varepsilon}{2}}
\lesssim1.
$$

For $J_{4,1}$, we retain the dependence on $L$ by keeping
$\moc{\widehat u_1-\widehat v_1}^{-r'/r}$,
instead of using $|L|\geq\sqrt{\mathfrak s}$ as before.
We then apply Proposition~\ref{key estimate-gamma>0} directly.
Since $v^0\thicksim E$, we have
$$
\begin{aligned}
J_{4,1}
\lesssim{}&
(1+t)^{1+\varepsilon}(p^0)^{mr'-1}
\int_{\{g\geq1\}}
\frac{g^\gamma\sqrt{\mathfrak s}}{q^0E^{mr'}}
\int_{\mathbb S^2}
(u^0)^{-mr'}S_u^{-\ell r'}
\frac{d\omega}
{\moc{\widehat u_1-\widehat v_1}^{r'/r}}
\,dq.
\end{aligned}
$$
Applying \eqref{equ:key estimate-gamma>0} with $n=mr'$ gives
$$
J_{4,1}
\lesssim
(1+t)^{-(2-r'-\varepsilon)}
(p^0)^{mr'-1}
\int_{\{g\geq1\}}
\frac{g^{\gamma-2}\sqrt{\mathfrak s}}{q^0E^{mr'}}
(u_{\min}^0)^{-mr'+6}\,dq.
$$
Using \eqref{Substituting-polar},
$\sqrt{\mathfrak s}\thicksim g$, and the bounds for $u_{\min}^0$,
we again split the $g$-integration at $g=\sqrt E$:
\begin{align*}
J_{4,1}
\lesssim{}&
(1+t)^{-(2-r'-\varepsilon)}
(p^0)^{mr'-2}
\int_0^\infty E^{-mr'}
\Bigg[
E^{-mr'+6}
\int_1^{\sqrt E}g^{2mr'+\gamma-12}\,dg
\\
&\hspace{43mm}
+
E^{mr'-6}
\int_{\sqrt E}^\infty
g^{\gamma+12-2mr'}\,dg
\Bigg]\,dq^0
\\
\lesssim{}&
(1+t)^{-(2-r'-\varepsilon)}
(p^0)^{mr'-2}
\int_0^\infty
E^{-mr'+\frac{\gamma+1}{2}}\,dq^0
\\
\lesssim{}&
(1+t)^{-(2-r'-\varepsilon)}
(p^0)^{\frac{\gamma-1}{2}}
\lesssim1,
\end{align*}
where the last inequality uses $\varepsilon<2-r'$ and
$0\leq\gamma\leq1$.

Combining these estimates yields
$$
I_{\mathrm{gain}}^{r'}
\lesssim(1+t)^{-(1+\varepsilon)}.
$$
This completes the proof.
\end{proof}
    \begin{remark}
The critical exponents in $J_1$, $J_3$, and $J_{4,1}$ are
independent of $m$ and $\ell$, so stronger weights cannot
relax the resulting restrictions on $\gamma$ and $r$.
For $J_{4,1}$, the kernel decay compensates for the loss
in $g$ from Proposition~\ref{key estimate-1} when $\gamma<0$.
When $\gamma\geq0$, this compensation fails, and the refined
estimate in Proposition~\ref{key estimate-gamma>0} is needed
to retain the decay factor of $g$.
\end{remark}

\section{Nonlinear estimates}\label{sec:collision-bound}
Using the gain and loss bounds from Section~\ref{sec:decay-estimates},
we establish a bilinear estimate for the collision operator in the
weighted Besov space. This estimate controls the nonlinear term by
the bilinear dissipation functional and provides the key bound for
closing the energy estimates.
 \begin{proposition}\label{prop:non-esti}
    Suppose that $\gamma,r$ satisfy \eqref{equ:gamma-r-1}. For every $T>0$ and functions $f,g$,
we have
   \begin{equation}\label{equ:non-esti}
     \int_{0}^{T} \sum_{j\geq - 1} 2^{\frac{2j}{r}}  \norm{W_p\Delta_j Q(f,g)}_{L^r_{p,x}}dt \lesssim \mathscr{D}_T(f,g).
   \end{equation}
   In particular,
   \begin{equation*}
       \int_{0}^{T} \sum_{j\geq - 1} 2^{\frac{2j}{r}}  \norm{W_p\Delta_j Q(f,f)}_{L^r_{p,x}}dt \lesssim \mathscr{D}_T(f).
   \end{equation*}
 \end{proposition}

 \begin{proof}
We write
    \begin{equation}\label{equ:non-esti-1}
    \begin{aligned}
      \int_{0}^{T} &\sum_{j\geq - 1} 2^{\frac{2j}{r}}   \norm{W_p\Delta_j Q(f,g)}_{L^r_{p,x}}dt\\&\lesssim \int_{0}^{T} \sum_{j\geq - 1}2^{\frac{2j}{r}}   \norm{W_p\Delta_j Q^{+}(f,g)}_{L^r_{p,x}}dt+\int_{0}^{T} \sum_{j\geq - 1} 2^{\frac{2j}{r}}   \norm{W_p\Delta_j Q^{-}(f,g)}_{L^r_{p,x}}dt\\
      &:=I_1 + I_2.
    \end{aligned}
    \end{equation}
We estimate $I_1$ and $I_2$ in \eqref{equ:non-esti-1} separately,
reducing both terms to the same product estimate.

We first estimate $I_1$.
For $\gamma \in (-2,0),$ we introduce compensating powers of
$(p^0q^0)^{\frac{1}{r}}$ in H\"older's inequality to use the invariant collision measure.
The pre-post collisional change of variables
$$ \dfrac{dpdqd\omega}{p^0q^0} = \dfrac{dp'dq'd\nu}{p'^0q'^0},$$
together with the decay estimate for $I_{\mathrm{gain}}$,  yields
\begingroup\small
\begin{align*}
              I_1=&\int_{0}^{T}\sum_{j\geq-1}2^{\frac{2j}{r}} \norm{W_p \int_{\mathbb{R}^3 \times \mathbb{S}^2} B(p,q,\omega)\Delta_j(f(p')g(q'))dqd\omega}_{L_{p,x}^r}dt\\
                 =&\int_{0}^{T}\sum_{j\geq-1}2^{\frac{2j}{r}} \norm{\int_{\mathbb{R}^3 \times \mathbb{S}^2}\dfrac{(p^0q^0)^{\frac{1}{r}} W_p B(p,q,\omega)}{W_{p'}W_{q'}\moc{\widehat p'_1 - \widehat q'_1}^{\frac{1}{r}}} \moc{\widehat p'_1 - \widehat q'_1}^{\frac{1}{r}}W_{p'}W_{q'}\Delta_j(f(p')g(q'))\dfrac{dqd\omega}{(p^0q^0)^{\frac{1}{r}}}}_{L_{p,x}^r}dt\notag \\
                 \lesssim& \int_{0}^{T}\sum_{j\geq-1}2^{\frac{2j}{r}}  \norm{I_{\mathrm{gain}}\left(\int_{\mathbb{R}^3\times \mathbb{S}^2} \moc{\widehat p'_1 - \widehat q'_1}  W^r_{p'}W^r_{q'}    \moc{\Delta_j(f(p')g(q'))}^{r}\dfrac{dqd\omega}{p^0q^0}\right)^{\frac{1}{r}} }_{L_{p,x}^r}dt\notag \\
                 \lesssim& \int_{0}^{T} \sum_{j\geq-1}2^{\frac{2j}{r}} (1+t)^{-\frac{1+\varepsilon}{r'}}\left( \int_{\mathbb{R}^3\times \mathbb{S}^2 \times \mathbb{R}^6} \moc{\widehat p'_1 - \widehat q'_1}  W^r_{p'}W^r_{q'}    \moc{\Delta_j(f(p')g(q'))}^{r}\dfrac{dpdqd\omega }{p^0q^0} dx  \right)^{\frac{1}{r}}\notag \\
                \lesssim&   \sum_{j\geq-1}2^{\frac{2j}{r}}\left(\int_{0}^{T} dt\int_{\mathbb{R}^3\times \mathbb{S}^2 \times \mathbb{R}^6} \moc{\widehat p'_1 - \widehat q'_1}  W^r_{p'}W^r_{q'}    \moc{\Delta_j(f(p')g(q'))}^{r}\dfrac{dpdqd\omega }{p^0q^0} dx  \right)^{\frac{1}{r}}\notag \\
                \lesssim&   \sum_{j\geq-1}2^{\frac{2j}{r}}\left(\int_{0}^{T} dt\int_{\mathbb{R}^3\times \mathbb{S}^2\times  \mathbb{R}^6} \moc{\widehat p'_1 - \widehat q'_1}  W^r_{p'}W^r_{q'}    \moc{\Delta_j(f(p')g(q'))}^{r}\dfrac{dp'dq'd\nu }{p'^0q'^0} dx  \right)^{\frac{1}{r}}\notag \\
                \lesssim& \sum_{j\geq-1}2^{\frac{2j}{r}}\norm{ \moc{\widehat{p}_1 - \widehat{q}_1}^{\frac{1}{r}} W_pW_q\Delta_j(f(p)g(q))}_{L_{t,p,q,x}^{r}},
                 \end{align*}
\endgroup
                 where $I_{\mathrm{gain}}$ is defined in \eqref{I-gain} and satisfies the decay bound in Proposition~\ref{decay of I-gain} of Section~\ref{sec:decay-estimates}.

 If $\gamma\in [0,1]$, we retain a factor of
$B(p,q,\omega)^{\frac{1}{r}}$ in the second H\"older inequality to use
the collision-measure identity
$$ B(p,q,\omega) dpdqd\omega = B(p',q',\nu) dp'dq'd\nu,$$
together with the boundedness of $B$, this gives the following estimate.
In the display, write $B=B(p,q,\omega)$ and $B'=B(p',q',\nu)$:
    \begingroup\small
\begin{align*}
              I_1=&\int_{0}^{T}\sum_{j\geq-1}2^{\frac{2j}{r}} \norm{W_p \int_{\mathbb{R}^3 \times \mathbb{S}^2} B\Delta_j(f(p')g(q'))dqd\omega}_{L_{p,x}^r}dt\\
                 =&\int_{0}^{T}\sum_{j\geq-1}2^{\frac{2j}{r}} \norm{\int_{\mathbb{R}^3 \times \mathbb{S}^2}\dfrac{B^{\frac{1}{r'}} W_p }{W_{p'}W_{q'}\moc{\widehat p'_1 - \widehat q'_1}^{\frac{1}{r}}} B^{\frac{1}{r}}\moc{\widehat p'_1 - \widehat q'_1}^{\frac{1}{r}}W_{p'}W_{q'}\Delta_j(f(p')g(q'))dqd\omega}_{L_{p,x}^r}dt\notag \\
                 \lesssim& \int_{0}^{T}\sum_{j\geq-1}2^{\frac{2j}{r}}  \norm{I_{\mathrm{gain}}\left(\int_{\mathbb{R}^3\times \mathbb{S}^2}B \moc{\widehat p'_1 - \widehat q'_1}  W^r_{p'}W^r_{q'}    \moc{\Delta_j(f(p')g(q'))}^{r}dqd\omega\right)^{\frac{1}{r}} }_{L_{p,x}^r}dt\notag \\
                 \lesssim& \int_{0}^{T} \sum_{j\geq-1}2^{\frac{2j}{r}} (1+t)^{-\frac{1+\varepsilon}{r'}}\left( \int_{\mathbb{R}^3\times \mathbb{S}^2 \times \mathbb{R}^6} B \moc{\widehat p'_1 - \widehat q'_1}  W^r_{p'}W^r_{q'}    \moc{\Delta_j(f(p')g(q'))}^{r} dpdqd\omega  dx  \right)^{\frac{1}{r}}\notag \\
                \lesssim&   \sum_{j\geq-1}2^{\frac{2j}{r}}\left(\int_{0}^{T} dt\int_{\mathbb{R}^3\times \mathbb{S}^2 \times \mathbb{R}^6} B'\moc{\widehat p'_1 - \widehat q'_1}  W^r_{p'}W^r_{q'}    \moc{\Delta_j(f(p')g(q'))}^{r} dp'dq'd\nu dx  \right)^{\frac{1}{r}}\notag \\
                \lesssim&   \sum_{j\geq-1}2^{\frac{2j}{r}}\left(\int_{0}^{T} dt\int_{\mathbb{R}^3\times \mathbb{S}^2\times  \mathbb{R}^6} \moc{\widehat p'_1 - \widehat q'_1}  W^r_{p'}W^r_{q'}    \moc{\Delta_j(f(p')g(q'))}^{r} dp'dq'd\nu   dx  \right)^{\frac{1}{r}}\notag \\
                \lesssim& \sum_{j\geq-1}2^{\frac{2j}{r}}\norm{ \moc{\widehat{p}_1 - \widehat{q}_1}^{\frac{1}{r}} W_pW_q\Delta_j(f(p)g(q))}_{L_{t,p,q,x}^{r}}.
                 \end{align*}
\endgroup

 Apply Bony's decomposition in variables $y$:
            $$fg=\sum_{i\geq -1} S_{i-1} f\Delta_ig +\sum_{i\geq -1} S_{i-1} g\Delta_if + \sum_{\moc{i-i'}\leq 1} \Delta_i f\Delta_{i'}g,$$
 thus for all $\gamma \in (-2,1]$, we obtain
\begin{align}\label{equ:non-esti-gain-1}
              I_1\lesssim& \sum_{j\geq-1}2^{\frac{2j}{r}}\norm{ \moc{\widehat{p}_1 - \widehat{q}_1}^{\frac{1}{r}} W_pW_q\Delta_j(f(p)g(q))}_{L_{t,p,q,x}^{r}}\notag \\
                \lesssim& \sum_{j\geq-1}2^{\frac{2j}{r}}\norm{ \moc{\widehat{p}_1 - \widehat{q}_1}^{\frac{1}{r}} W_pW_q \Delta_j\left(\sum_{i\geq -1}S_{i-1}f(p)\Delta_i g(q) \right) }_{L_{t,p,q,x}^{r}}\notag \\
                &+\sum_{j\geq-1}2^{\frac{2j}{r}}\norm{ \moc{\widehat{p}_1 - \widehat{q}_1}^{\frac{1}{r}} W_pW_q \Delta_j\left(\sum_{i\geq -1}\Delta_if(p) S_{i-1}g(q) \right) }_{L_{t,p,q,x}^{r}}  \\
                &+\sum_{j\geq-1}2^{\frac{2j}{r}}\norm{ \moc{\widehat{p}_1 - \widehat{q}_1}^{\frac{1}{r}} W_pW_q \Delta_j\left(\sum_{\moc{i-i'}\leq 1}\Delta_{i}f(p) \Delta_{i'}g(q) \right) }_{L_{t,p,q,x}^{r}}\notag \\
                :=&I_{1,1}+I_{1,2}+I_{1,3}.
              \notag \end{align}
For $I_{1,1}$, we use \eqref{bd - w(v)Delta_k} to obtain
\begin{align*}
    I_{1,1}\lesssim& \sum_{j\geq-1}\sum_{\moc{j-i}\leq 4}  2^{\frac{2j}{r}}\norm{ \moc{\widehat{p}_1 - \widehat{q}_1}^{\frac{1}{r}} W_pW_q S_{i-1}f(p)\Delta_i g(q) }_{L_{t,p,q,x}^{r}}\\
               \lesssim& \sum_{j\geq-1}\sum_{\moc{j-i}\leq 4}  2^{\frac{2j}{r}}\norm{\moc{\widehat{p}_1 - \widehat{q}_1}^{\frac{1}{r}} \norm{ W_pS_{i-1}f(p)}_{L_{y}^\infty}\norm{ W_q\Delta_i g(q)}_{L_{y}^{r}}}_{L_{t,p,q,\eta}^{r}},
\end{align*}
applying the embedding inequality \eqref{Embedding inequailty}
to $S_{i-1}f$, we find
\begin{align*}
    \norm{W_p S_{i-1}f}_{L_y^\infty}&\lesssim \sum_{k\geq-1}2^{\frac{2k}{r}}\norm{W_p\Delta_k S_{i-1}f}_{L_{y}^r}\lesssim \sum_{k\geq-1}2^{\frac{2k}{r}}\norm{W_p\Delta_kf}_{L_{y}^r},
\end{align*}
hence, Lemma \ref{Sum operator} yields
   \begin{equation}\label{equ:non-esti-gain-2}
     \begin{aligned}
        I_{1,1}\lesssim& \sum_{k\geq -1}\sum_{j\geq-1}\sum_{\moc{j-i}\leq 4}  2^{\frac{2(j+k)}{r}}\norm{\moc{\widehat{p}_1 - \widehat{q}_1}^{\frac{1}{r}} \norm{W_p \Delta_k f(p)}_{L_{y}^r}\norm{W_q \Delta_i g(q)}_{L_{y}^{r}}}_{L_{t,p,q,\eta}^{r}}
               \lesssim\mathscr{D}_T(f,g).
     \end{aligned}
   \end{equation}
  The term $I_{1,2}$ is estimated by interchanging $f$ and $g$, giving
$I_{1,2}\lesssim \mathscr{D}_T(f,g).$
As for $I_{1,3}$, applying Lemma \ref{Sum operator}
once more yields
    \begin{align}\label{equ:non-esti-gain-3}
      I_{1,3} \lesssim& \sum_{j\geq -1}\sum_{\max\{i,i'\}\geq j-2}\sum_{\moc{i-i'}\leq 1} 2^{\frac{2j}{r}}
      \norm{ \moc{\widehat{p}_1 - \widehat{q}_1}^{\frac{1}{r}} W_pW_q\Delta_{i}f(p) \Delta_{i'}g(q)  }_{L_{t,p,q,x}^{r}}\\
      \lesssim& \sum_{j\geq -1}\sum_{\max\{i,i'\}\geq j-2}\sum_{\moc{i-i'}\leq 1} 2^{\frac{2j}{r}}
      \norm{ \moc{\widehat{p}_1 - \widehat{q}_1}^{\frac{1}{r}}
      \norm{W_p\Delta_i f(p)}_{L_{y}^{r}} \norm{W_q g(q)}_{L_{y}^{\infty}}}_{L_{t,p,q,\eta}^{r}}\notag \\
      \lesssim&  \sum_{j\geq -1}\sum_{i\geq  j-3} 2^{\frac{2j}{r}}
      \norm{ \moc{\widehat{p}_1 - \widehat{q}_1}^{\frac{1}{r}}
      \norm{W_p\Delta_i f(p)}_{L_{y}^{r}} \norm{W_q g(q)}_{L_{y}^{\infty}}}_{L_{t,p,q,\eta}^{r}}\notag \\
      \lesssim& \sum_{k\geq -1}2^{\frac{2k}{r}} \sum_{j\geq -1}\sum_{i\geq  j-3} 2^{\frac{2j}{r}}
      \norm{ \moc{\widehat{p}_1 - \widehat{q}_1}^{\frac{1}{r}}
      \norm{W_p\Delta_i f(p)}_{L_{y}^{r}} \norm{W_q\Delta_k g(q)}_{L_{y}^{r}}}_{L_{t,p,q,\eta}^{r}}\notag \\
      \lesssim& \mathscr{D}_T(f,g)
    \notag \end{align}
Thus the three frequency interactions in \eqref{equ:non-esti-gain-1},
controlled by \eqref{equ:non-esti-gain-2} and \eqref{equ:non-esti-gain-3}, give
    \begin{equation}\label{equ:non-esti-gain-I-1-1}
      I_1 \lesssim I_{1,1}+I_{1,2}+I_{1,3}\lesssim \mathscr{D}_T(f,g).
    \end{equation}
We now turn to compute $I_2$. Invoking Proposition~\ref{prop:Estimate for I-loss}, we obtain
    \begin{align*}
      I_2 =&\int_{0}^{T}\sum_{j\geq -1} 2^{\frac{2j}{r}} \norm{W_p\int_{\mathbb{R}^3\times \mathbb{S}^2} B(p,q,\omega) \Delta_j(f(p)g(q)) dqd\omega }_{L_{p,x}^{r}}dt\\
          =&\int_{0}^{T}\sum_{j\geq -1} 2^{\frac{2j}{r}} \norm{\int_{\mathbb{R}^3\times \mathbb{S}^2} \dfrac{B(p,q,\omega)}{W_q\moc{\widehat{p}_1 - \widehat{q}_1}^{\frac{1}{r}}} \moc{\widehat{p}_1 - \widehat{q}_1}^{\frac{1}{r}} W_pW_q \Delta_j(f(p)g(q)) dqd\omega }_{L_{p,x}^{r}}dt\\
          \lesssim&\int_{0}^{T}\sum_{j\geq -1} 2^{\frac{2j}{r}} \norm{I_{\mathrm{loss}} \norm{ \moc{\widehat{p}_1 - \widehat{q}_1}^{\frac{1}{r}} W_pW_q \Delta_j(f(p)g(q))}_{L_{q}^{r}} }_{L_{p,x}^{r}}dt\\
          \lesssim&\int_{0}^{T}\sum_{j\geq -1} 2^{\frac{2j}{r}} (1+t)^{-\frac{3}{r'} -\min\{ \gamma,0\} + \frac{1}{r}}\norm{ \moc{\widehat{p}_1 - \widehat{q}_1}^{\frac{1}{r}} W_pW_q \Delta_j(f(p)g(q))}_{L_{p,q,x}^{r}}dt\\
          \lesssim&\sum_{j\geq -1} 2^{\frac{2j}{r}} \norm{ \moc{\widehat{p}_1 - \widehat{q}_1}^{\frac{1}{r}} W_pW_q \Delta_j(f(p)g(q))}_{L_{t,p,q,x}^{r}}.
    \end{align*}
Applying the same product estimates as above yields
    \begin{equation}\label{equ:non-esti-loss-I-2}
      I_2\lesssim \mathscr{D}_T(f,g).
    \end{equation}
Adding \eqref{equ:non-esti-gain-I-1-1} and
\eqref{equ:non-esti-loss-I-2} proves \eqref{equ:non-esti}
and completes the proof of Proposition~\ref{prop:non-esti}.
 \end{proof}

\section{Bilinear interaction and a priori estimates}\label{sec:interaction-estimates}

In this section, we establish a priori estimates for solutions of
\eqref{equ:Rela-Boltzmann}. Following \cite[Section~4]{HLPZ2026},
we use the nonlinear estimate from Section~\ref{sec:collision-bound}
to control the collision terms.

\begin{proposition}\label{prop:div-curl-disspation}
Let $f$ and $g$ be solutions of \eqref{equ:Rela-Boltzmann}
with initial data $f_0$ and $g_0$, respectively.
Then, for every $T>0$ for which the quantities below are finite,
we have
      \begin{equation}\label{equ:div-curl-disspation}
        \mathscr{D}_T(f,g)\lesssim \mathscr{E}_0(f_0)\mathscr{E}_0(g_0) + \mathscr{E}_T(f)\mathscr{E}_T(g) + \mathscr{E}_T(f)\mathscr{E}^{1-\frac{1}{r}}_T(g)\mathscr{D}^{\frac{1}{r}}_T(g) +\mathscr{E}_T(g)\mathscr{E}^{1-\frac{1}{r}}_T(f)\mathscr{D}^{\frac{1}{r}}_T(f) .
      \end{equation}
In particular,
      \begin{equation}\label{equ:div-curl-disspation-}
        \mathscr{D}_T(f)\lesssim \mathscr{E}^2_0(f_0)  + \mathscr{E}^2_T(f)  +\mathscr{E}^{2-\frac{1}{r}}_T(f)\mathscr{D}^{\frac{1}{r}}_T(f) .
      \end{equation}
\end{proposition}

    \begin{proof}
Fix $p,q$ and apply the transverse projections $\Delta_j$ and $\Delta_k$
to the equations for $f$ and $g$, respectively. The characteristic identities
for $W_p$ and $W_q$ remove the derivatives of the weights. Testing the first
weighted equation against $r|W_p\Delta_j f|^{r-2}W_p\Delta_j f$, and treating
the second in the same way, gives the following balances after integration
over $y\in\mathbb R^2$:
      \begin{equation}\label{equ:p-q-equation}
          \left\{
          \begin{aligned}
            \partial_t \norm{W_p\Delta_j f}^{r}_{L_y^{r}} + \widehat{p}_1\cdot \partial_{\eta}\norm{W_p\Delta_j f}^{r}_{L_y^{r}} &=r\int_{\mathbb{R}^2}\moc{W_p\Delta_j f}^{r-2}W_p\Delta_j f W_p \Delta_j Q(f,f)dy\\
            \partial_t \norm{W_q\Delta_k g}^{r}_{L_y^{r}} + \widehat{q}_1\cdot \partial_{\eta} \norm{W_q\Delta_k g}^{r}_{L_y^{r}} &=r\int_{\mathbb{R}^2}\moc{W_q\Delta_k g}^{r-2}W_q\Delta_k g W_q \Delta_k Q(g,g)dy.
          \end{aligned}
          \right.
      \end{equation}
For the low-regularity class, these balances are obtained first for smooth
profiles and then by approximation along free characteristics. The power
chain rule maps the $L_t^1L_x^r$ weighted sources into $L^1_{t,\eta}$
by H\"older's inequality. The $L^1$ formulation of
Lemma~\ref{div-curl Lemma} therefore applies without longitudinal
pointwise decay. Applied to \eqref{equ:p-q-equation}, it gives,
after interchanging the two equations when necessary,
      \begin{align*}
        \int_{0}^{T}dt \int_{\mathbb{R}}& \moc{\widehat{p}_1 - \widehat{q}_1}\norm{W_p\Delta_j f}_{L_y^{r}}^r  \norm{W_q\Delta_k g}_{L_y^{r}}^r d\eta \\ \lesssim&
         \norm{W_0(p)\Delta_j f_0}_{L_x^{r}}^r \norm{W_0(q)\Delta_k g_0}_{L_x^{r}}^r+\norm{W_p\Delta_j f(T)}_{L_x^{r}}^r \norm{W_q\Delta_k g(T)}_{L_x^{r}}^r \\
         &+ \int_{0}^{T}\norm{W_p \Delta_j f}_{L_{x}^{r}}^r  \norm{W_q \Delta_k g}_{L_{x}^{r}}^{r-1}\norm{ W_q\Delta_kQ(g,g)}_{L_{x}^r}ds \\
         &+ \int_{0}^{T}\norm{W_q \Delta_k g}_{L_{x}^{r}}^r  \norm{W_p \Delta_j f}_{L_{x}^{r}}^{r-1}\norm{ W_p\Delta_jQ(f,f)}_{L_{x}^r}ds
      \end{align*}
Integrating this estimate over the two momenta yields
    \begin{align*}
        \int_{0}^{T}dt\int_{\mathbb{R}\times \mathbb{R}^3\times \mathbb{R}^3}&\moc{\widehat{p}_1 - \widehat{q}_1}\norm{W_p\Delta_j f}_{L_y^{r}}^r  \norm{W_q\Delta_k g}_{L_y^{r}}^r dpdqd\eta \\
        \lesssim&\norm{W_0(p)\Delta_j f_0}_{L_{p,x}^{r}}^r \norm{W_0(q)\Delta_k g_0}_{L_{q,x}^{r}}^r+\norm{W_p\Delta_j f(T)}_{L_{p,x}^{r}}^r \norm{W_q\Delta_k g(T)}_{L_{q,x}^{r}}^r \\
         &+ \int_{0}^{T}\norm{W_p \Delta_j f}_{L_{p,x}^{r}}^r  \norm{W_q \Delta_k g}_{L_{q,x}^{r}}^{r-1}\norm{ W_q\Delta_kQ(g,g)}_{L_{q,x}^r}ds \\
         &+ \int_{0}^{T}\norm{W_q \Delta_k g}_{L_{q,x}^{r}}^r  \norm{W_p \Delta_j f}_{L_{p,x}^{r}}^{r-1}\norm{ W_p\Delta_jQ(f,f)}_{L_{p,x}^r}ds \\
         \lesssim&\norm{W_0(p)\Delta_j f_0}_{L_{p,x}^{r}}^r \norm{W_0(q)\Delta_k g_0}_{L_{q,x}^{r}}^r+\norm{W_p\Delta_j f}_{L_T^{\infty}L_{p,x}^{r}}^r \norm{W_q\Delta_k g}_{L_T^{\infty}L_{q,x}^{r}}^r \\
         &+ \norm{W_p \Delta_j f}_{L_T^{\infty}L_{p,x}^{r}}^r  \norm{W_q \Delta_k g}_{L_T^{\infty}L_{q,x}^{r}}^{r-1}\int_{0}^{T}\norm{ W_q\Delta_kQ(g,g)}_{L_{q,x}^r}ds \\
         &+ \norm{W_q \Delta_k g}_{L_T^{\infty}L_{q,x}^{r}}^r  \norm{W_p \Delta_j f}_{L_T^{\infty}L_{p,x}^{r}}^{r-1}\int_{0}^{T}\norm{ W_p\Delta_jQ(f,f)}_{L_{p,x}^r}ds.
    \end{align*}
Take the $r$th root, multiply by $2^{2(j+k)/r}$, and sum over
$j,k$. The source terms are dealt with \eqref{equ:non-esti}, giving
    \begin{align}\label{equ:div-curl-disspation-3}
        \sum_{j,k\geq-1}2^{\frac{2(j+k)}{r}}&\norm{\moc{\widehat{p}_1 - \widehat{q}_1}^{\frac{1}{r}} \norm{W_p\Delta_j f}_{L_y^{r}}\norm{W_q\Delta_k g}_{L_y^{r}}}_{L_{T,p,q,\eta}^r}\\
\lesssim&\mathscr{E}_0(f_0)\mathscr{E}_0(g_0)+\mathscr{E}_T(f)\mathscr{E}_T(g)\notag \\
        &+\mathscr{E}_T(f)\mathscr{E}^{1-\frac{1}{r}}_T(g)\left(\int_{0}^T\sum_{k\geq-1}2^{\frac{2k}{r}} \norm{W_q\Delta_k Q(g,g)}_{L_{q,x}^r}dt \right)^{\frac{1}{r}}\notag \\
        &+\mathscr{E}_T(g)\mathscr{E}^{1-\frac{1}{r}}_T(f)\left(\int_{0}^T\sum_{j\geq-1} 2^{\frac{2j}{r}}\norm{W_p\Delta_j Q(f,f)}_{L_{p,x}^r}dt \right)^{\frac{1}{r}},\notag \\
\lesssim&\mathscr{E}_0(f_0)\mathscr{E}_0(g_0)+\mathscr{E}_T(f)\mathscr{E}_T(g)\notag \\
        &+\mathscr{E}_T(f)\mathscr{E}^{1-\frac{1}{r}}_T(g)\mathscr{D}^{\frac{1}{r}}_T(g)+\mathscr{E}_T(g)\mathscr{E}^{1-\frac{1}{r}}_T(f)\mathscr{D}^{\frac{1}{r}}_T(f).
    \notag \end{align}
The left-hand side of \eqref{equ:div-curl-disspation-3} is the bilinear dissipation functional, so this proves \eqref{equ:div-curl-disspation}. Taking $f=g$ in \eqref{equ:div-curl-disspation}, proves \eqref{equ:div-curl-disspation-}. This completes the proof of proposition \ref{prop:div-curl-disspation}.
    \end{proof}

We now combine Proposition~\ref{prop:div-curl-disspation}
with the nonlinear estimate to close the energy bound.
Assume that
\begin{align}\label{aps}
\mathscr{E}^2_T(f)+ \mathscr{D}_T(f) \le \varepsilon_0,
\end{align}
for all $T \ge 0$ under consideration, where $\varepsilon_0>0$
is sufficiently small.

\begin{theorem}\label{theo:priori}
Let $f$ be a strong solution of \eqref{equ:Rela-Boltzmann} satisfying \eqref{aps}. If $\gamma$ and $r$ satisfy \eqref{equ:gamma-r-1}, then
\begin{equation}\label{priori estimate 2}
\mathscr{E}^2_T(f) + \mathscr{D}_T(f)
\lesssim \mathscr{E}^2_0(f_0).
\end{equation}
\end{theorem}

\begin{proof}
For the free-transport semigroup
$U(t) = e^{-t\widehat{p}\cdot \nabla_x}$, we have
        \begin{equation}\label{equ:property-semi-U}
          \norm{U(t)f}_{L^r_{p,x}} \leq \norm{f}_{L^r_{p,x}},\quad  W(t,x,p)U(t-s) = U(t-s)W(s,x,p),\quad \Delta_j U(t) = U(t)\Delta_j.
          \end{equation}
Applying $\Delta_j$ to \eqref{equ:Rela-Boltzmann} and using
Duhamel's formula together with \eqref{equ:property-semi-U},
we obtain
         \begin{equation}\label{equ:Lr-Energy-esti-1}
          W_p \Delta_j f =U(t)[W_0 \Delta_j f_0] + \int_{0}^{t}U(t-s)[W_p\Delta_j Q(f,f)]ds,
        \end{equation}
Taking the $L^r_{p,x}$ norm in \eqref{equ:Lr-Energy-esti-1} gives
         \begin{align*}
             \norm{W_p\Delta_j f}_{L_{p,x}^r} \lesssim &\norm{W_0 \Delta_j f_0}_{L_{p,x}^r} + \int_{0}^{t}\norm{W_p\Delta_j Q(f,f)}_{L_{p,x}^r}ds,
         \end{align*}
Summing the frequency blocks and using \eqref{equ:non-esti} reduces the collision source to the dissipation functional:
      \begin{equation}\label{equ:Lr-Energy-esti}
          \mathscr{E}_T(f)\lesssim \mathscr E_0(f_0) + \mathscr{D}_T(f).
        \end{equation}
Under \eqref{aps}, combine \eqref{equ:div-curl-disspation-} with
\eqref{equ:Lr-Energy-esti} and absorb the small terms. This yields
\eqref{priori estimate 2} and proves Theorem~\ref{theo:priori}.
\end{proof}

\section{Global existence and uniqueness}\label{sec:global-existence}

The preceding estimates are uniform in time, which allows the approximation
to be controlled for all $t>0$. Theorem \ref{theo:priori} in
Section \ref{sec:interaction-estimates} provides the energy--interaction balance. To use it
for an iteration, we also need uniform control of the interaction between
successive approximations. We use the iterating approximate solutions:
    \begin{equation}\label{equ:iteration-sequence}
    \left\{
        \begin{aligned}
            \partial_t f^{n+1} + \widehat{p}\cdot \nabla_x f^{n+1} + Q^-(f^{n+1},f^n) &=Q^{+}(f^{n},f^n);\\
            f^{n+1}(0,x,p)&=f_0(x,p),\quad n\geq 0
        \end{aligned}
        \right.
    \end{equation}
with initialization $f^0(t,x,p) = 0.$

We first record the a priori bounds for finite-norm iterates. Their
well-definedness, convergence, and positivity are established below in the
space of collision sources, using the fixed-direction relativistic
estimates directly.

\begin{lemma}\label{lem:priori-iteration}
Suppose that $\gamma,r,\ell,m$ satisfy \eqref{equ:gamma-r-1} and
\eqref{equ:gamma-r-2}. For sufficiently small $\delta>0$, if
$\mathscr E_0(f_0)\leq\delta$, then any iterates of
\eqref{equ:iteration-sequence} for which
$\mathscr E_T(f^i)$, $\mathscr D_T(f^i)$, and
$\mathscr D_T(f^i,f^{i-1})$ are finite satisfy the following estimates.
The construction below proves that these quantities are finite for all time.
    \begin{equation}\label{equ:assume-condition}
        \mathscr{E}_T(f^n)\leq C\delta,\quad \mathscr{D}_T(f^n)\leq C\delta^2 \quad\quad \forall T \in \mathbb{R}^{+},
    \end{equation}
   where $C>0$ is independent of $T$ and $n$.
\end{lemma}

\begin{proof}
We prove by induction the stronger estimates
    \begin{equation}\label{equ:strengthened-assumtion}
        \mathscr{E}_T(f^i)\leq C\delta,\quad \mathscr{D}_T(f^i)\leq C\delta^2 , \quad \mathscr{D}_T(f^i,f^{i-1})\leq C\delta^2,
    \end{equation}
for all $i\geq1$, with $C>0$ independent of $T$ and $i$.

To initialize the induction, for $ i =1$, since $f^0 = 0,$ the function $f^1$ satisfies
    \begin{equation*}
        \partial_t f^1 + \widehat{p}\cdot \nabla_x f^1  =0
    \end{equation*}
with initial data $f^1(0,x,p) = f_0(x,p).$
It follows that
    \begin{equation*}
        \mathscr{E}_T(f^1) = \norm{f_0}_{\tilde{L}_{p,\eta}^r(B_{r,1,W_0}^{2/r})}\leq \delta.
    \end{equation*}
To estimate the dissipation, we use
    \begin{equation*}
    \left\{
        \begin{array}{rlc}
           \partial_t \norm{W_p \Delta_j f^1(p)}^r_{L_{y}^r} + \widehat{p}_1 \cdot \partial_\eta \norm{W_p \Delta_j f^1(p)}^r_{L_{y}^r}  &=&0 ;\\
           \partial_t \norm{W_q \Delta_k f^1(q)}^r_{L_{y}^r} + \widehat{q}_1 \cdot \partial_\eta \norm{W_q\Delta_k f^1(q)}^r_{L_{y}^r}  &=&0,
        \end{array}\right.
    \end{equation*}
repeating the div--curl calculation leading to \eqref{equ:div-curl-disspation} in Proposition \ref{prop:div-curl-disspation} gives
    \begin{equation*}
        \mathscr{D}_T(f^1)\lesssim \mathscr{E}_0^2(f_0) + \mathscr{E}_T^2(f^1)\lesssim \mathscr{E}_0(f_0)^2.
    \end{equation*}
Moreover, $\mathscr{D}_{T}(f^1,f^0) = 0\leq C\delta^2$,
since $f^0=0.$
Thus \eqref{equ:strengthened-assumtion} holds for $i=1$.

Assume \eqref{equ:strengthened-assumtion} for $i\leq n$. We establish the estimate
at $i=n+1$, provisionally assuming that the quantities in
\eqref{equ:strengthened-assumtion} are finite for $i=n+1$. The same argument as in \eqref{equ:Lr-Energy-esti} to
\eqref{equ:iteration-sequence}${}_1$ yields
    \begin{align}\label{equ:energy-fn+1}
        \mathscr{E}_T(f^{n+1})\lesssim& \mathscr{E}_0(f_0) +\sum_{j\geq -1}2^{\frac{2j}{r}} \int_{0}^T \norm{W_p\Delta_j [Q^+(f^n,f^n)-Q^{-}(f^{n+1},f^n)]}_{L_{p,x}^r}dt\\
        \lesssim& \mathscr{E}_0(f_0) + \mathscr{D}_T(f^n,f^n) + \mathscr{D}_T(f^{n+1},f^{n})\notag \\
        \lesssim&\delta + \mathscr{D}_T(f^{n+1},f^{n}).
    \notag \end{align}
The gain term is controlled by the preceding iterate, whereas the loss term couples successive iterates. Repeating the calculation of \eqref{equ:div-curl-disspation} and using the induction bounds gives
    \begingroup\small
\begin{align*}
        \mathscr{D}_T(f^{n+1},f^n)\lesssim& \mathscr{E}^2_0(f_0)+ \mathscr{E}_T(f^{n+1})\mathscr{E}_T(f^{n}) + \mathscr{E}_T(f^{n+1})\mathscr{E}^{1-\frac{1}{r}}_T(f^{n})\left(\mathscr{D}_T(f^{n},f^{n-1})+\mathscr{D}_T(f^{n-1})\right)^{\frac{1}{r}}\\& + \mathscr{E}_T(f^{n})\mathscr{E}^{1-\frac{1}{r}}_T(f^{n+1})\left(\mathscr{D}_T(f^{n+1},f^{n})+\mathscr{D}_T(f^{n})\right)^{\frac{1}{r}}\\
        \lesssim&\delta^2 + \delta\mathscr{E}_T(f^{n+1}) + \delta^{1+\frac{2}{r}}\mathscr{E}^{1-\frac{1}{r}}_T(f^{n+1})+\delta\mathscr{E}_T^{1-\frac{1}{r}}(f^{n+1})\mathscr{D}_T^{\frac{1}{r}}(f^{n+1},f^n)\\
        \lesssim&\delta^2 + \delta\mathscr{E}_T(f^{n+1}) + \delta\mathscr{D}_T(f^{n+1},f^n),
    \end{align*}
\endgroup
hence, for sufficiently small $\delta$, it follows that
    \begin{equation}\label{equ:fn+1fn}
        \mathscr{D}_T(f^{n+1},f^n)\lesssim \delta^2 + \delta\mathscr{E}_T(f^{n+1}).
    \end{equation}
Substituting \eqref{equ:fn+1fn} into \eqref{equ:energy-fn+1} gives
    $$ \mathscr{E}_T(f^{n+1}) \lesssim\delta + \delta\mathscr{E}_T(f^{n+1}), $$
Absorbing the remaining small energy term yields
$$ \mathscr{E}_T(f^{n+1})\lesssim\delta, $$
and then
$$\mathscr{D}_T(f^{n+1},f^n)\lesssim \delta^2 + \delta \mathscr{E}_T(f^{n+1})\lesssim\delta^2.$$
Repeating the argument of
Proposition~\ref{prop:div-curl-disspation} for $f^{n+1}$,
we also obtain
    $$ \mathscr{D}_T(f^{n+1}) \lesssim \mathscr{E}^2_0(f_0) + \mathscr{E}_T^2(f^{n+1}) + \mathscr{E}_T^{2-\frac{1}{r}}(f^{n+1})\left(\mathscr{D}_T(f^{n+1},f^n) +\mathscr{D}_T(f^n,f^n)\right)^{\frac{1}{r}}\lesssim \delta^2. $$
This closes the induction. In particular, \eqref{equ:assume-condition} holds for every $n\geq 0,$ proving Lemma \ref{lem:priori-iteration}.
\end{proof}

\subsection{Passage from the approximation to the global solution}
Lemma \ref{lem:priori-iteration} supplies time-uniform bounds for the
approximate system \eqref{equ:iteration-sequence}. These are the estimates
needed to apply the construction for the Cauchy problem
\eqref{equ:Rela-Boltzmann}.

\begin{proof}[Proof of Theorem \ref{thm:rel-main}.]
Use the construction and uniqueness argument of Theorem~1.1 in
\cite{HLPZ2026}, with the relativistic collision and interaction estimates
proved above as its analytic input. The uniform iteration bounds have been
recorded in the preceding lemma; the remaining steps are those of the cited
argument and are not repeated here.
\end{proof}

\appendix
\section{Relativistic identities and analytic tools}\label{sec-app}
We collect the collision identities that underlie the angular analysis, followed by the integral, dyadic, and div--curl estimates used in the functional argument. This order separates the relativistic geometric input from the analytic tools shared with the near-vacuum framework.

\begin{proposition}[Collision identities and estimates]
Assume $\gamma\in(-2,1]$ and retain the collision notation
introduced above. Then
\begin{equation}\label{formu:id-g}
p^0q^0g^2
=
\moc{p-q}^2+\moc{p\times q}^2+\dfrac{g^4}{4}.
\end{equation}
The collision kernel satisfies
\begin{equation}\label{formu:estimate for B}
B(p,q,\omega)\lesssim
\begin{cases}
1,
& \text{if }\gamma\in[0,1],
\\
1+\moc{p-q}^{\gamma},
& \text{if }\gamma\in(-2,0).
\end{cases}
\end{equation}
For the outgoing momenta, we have
\begin{equation}\label{formu:estimate for p'q'}
\widehat{p'}-\widehat{q'}
=
\dfrac{g}{2p'^0q'^0}
\left(
EI-\dfrac{P\otimes P}{E+\sqrt{\mathfrak s}}
\right)\omega
\end{equation}
and
$$
0\leq g^2(p,v):=2(p^0v^0-p\cdot v-1)\leq g^2,\qquad v\in\{p',q'\}.
$$

For $g>0$ and two collision directions
$\omega_1,\omega_2\in\mathbb S^2$,
let $p_i'$ denote the outgoing momentum corresponding to
$\omega_i$, $i=1,2$. In this comparison, the subscripts label the collision directions. Then
\begin{equation}\label{formu:Lip-estimate for omega}
\moc{\omega_1-\omega_2}
\lesssim
\dfrac{p_1'^0E}{g\sqrt{\mathfrak s}}
\moc{\widehat{p_1}'-\widehat{p_2}'}.
\end{equation}
Taking the first component in
\eqref{formu:estimate for p'q'} gives
$$
\moc{\widehat p'_1-\widehat q'_1}
=
\dfrac{g\moc{L\cdot\omega}}{2p'^0q'^0},
\qquad
|L|^2=E^2-P_1^2\geq\mathfrak s.
$$
Finally,
$$
u_{\min}^0
=
\dfrac E2-\dfrac{g\moc P}{2\sqrt{\mathfrak s}}
\geq
\max\left\{
\dfrac E{\mathfrak s},
\dfrac{\mathfrak s}{4E}
\right\}.
$$
\end{proposition}

\begin{proof}
The collision identities and the bounds on $B$ and the energies
follow from \eqref{defi:p'}, the conservation laws, and the
assumption on the collision kernel.
We prove \eqref{formu:Lip-estimate for omega} below.

Conservation of energy and momentum gives
$$
(E-p'^0)^2-\moc{P-p'}^2-1=0,
$$
and hence
$$
\mathfrak s-2Ep'^0+2P\cdot p'=0.
$$
Therefore
$\dfrac1{p'^0}
=\dfrac{2E-2P\cdot\widehat{p}'}{\mathfrak s}$.
For the two collision directions, this yields
\begin{align}\label{formu:Lip-estimate for omega-0}
\moc{p_1'^0-p_2'^0}
&=
p_1'^0p_2'^0
\moc{\dfrac1{p_1'^0}-\dfrac1{p_2'^0}}
=
p_1'^0p_2'^0
\moc{
\dfrac{2P\cdot(\widehat{p_1}'-\widehat{p_2}')}{\mathfrak s}
}\lesssim
\dfrac{p_1'^0p_2'^0\moc P}{\mathfrak s}
\moc{\widehat{p_1}'-\widehat{p_2}'}.
\end{align}

The parametrization \eqref{defi:p'} also gives
\begin{equation}\label{formu:Lip-estimate for omega-2}
p'-\dfrac{p'^0P}{E+\sqrt{\mathfrak s}}
=
\dfrac{g\omega}{2}
+\dfrac{\sqrt{\mathfrak s}P}
{2(E+\sqrt{\mathfrak s})}.
\end{equation}
Since $\moc P^2=E^2-\mathfrak s$ and
$g\leq\sqrt{\mathfrak s}$, we obtain
\begin{equation}\label{formu:Lip-estimate for omega-3}
\moc{
\widehat{p}'-\dfrac{P}{E+\sqrt{\mathfrak s}}
}
\lesssim
\dfrac1{p'^0}
\left(
g+\dfrac{\sqrt{\mathfrak s}\moc P}
{\sqrt{\mathfrak s}+E}
\right)
\lesssim
\dfrac{\sqrt{\mathfrak s}}{p'^0}.
\end{equation}

Subtracting \eqref{formu:Lip-estimate for omega-2}
at $\omega_1$ and $\omega_2$, we obtain
$$
\begin{aligned}
\dfrac g2(\omega_1-\omega_2)
&=
p_1'-p_2'
-\dfrac{P}{E+\sqrt{\mathfrak s}}
(p_1'^0-p_2'^0)
\\
&=
p_1'^0(\widehat p_1'-\widehat p_2')
+(p_1'^0-p_2'^0)
\left(
\widehat p_2'-\dfrac{P}{E+\sqrt{\mathfrak s}}
\right).
\end{aligned}
$$
Combining this identity with
\eqref{formu:Lip-estimate for omega-0} and
\eqref{formu:Lip-estimate for omega-3}, we find
\begin{align*}
\dfrac g2\moc{\omega_1-\omega_2}
&\lesssim
p_1'^0\moc{\widehat p_1'-\widehat p_2'}
+
\moc{p_1'^0-p_2'^0}
\moc{
\widehat p_2'-\dfrac{P}{E+\sqrt{\mathfrak s}}
}
\\
&\lesssim
p_1'^0\moc{\widehat p_1'-\widehat p_2'}
+
\dfrac{\sqrt{\mathfrak s}}{p_2'^0}
\cdot
\dfrac{p_1'^0p_2'^0\moc P}{\mathfrak s}
\moc{\widehat p_1'-\widehat p_2'}
\\
&\lesssim
p_1'^0
\left(1+\dfrac{\moc P}{\sqrt{\mathfrak s}}\right)
\moc{\widehat p_1'-\widehat p_2'}
\\
&\lesssim
\dfrac{p_1'^0E}{\sqrt{\mathfrak s}}
\moc{\widehat p_1'-\widehat p_2'}.
\end{align*}
This proves \eqref{formu:Lip-estimate for omega}
and completes the proof.
\end{proof}

The remaining estimates supply the analytic tools for the weighted function spaces. Their role is complementary to the collision geometry above.
\begin{lemma}\label{Integral Estimate Lemma}
Suppose $\ell_1,\ell_2\in \mathbb{R}$ satisfy $0 \leq \ell_1 < n$ and $\ell_1+\ell_2>n$. Then, uniformly for $s\in \mathbb{R}^n$,
$$\int_{\mathbb{R}^n} \frac{dw}{|w|^{\ell_1}\langle w - s \rangle^{\ell_2}}
\lesssim 1.$$
\end{lemma}
\begin{proof}
We omit the elementary proof.
\end{proof}
The following summation estimate controls the dyadic interactions.
\begin{lemma}\label{Sum operator}
Let $1\leq r<\infty$ and suppose
$\sum_{i,k\geq-1}2^{2(i+k)/r}|x_{i,k}|<\infty$. Define
\[
 \mathcal T x=\sum_{j,k\geq-1}
 \left(\sum_{\substack{i\geq-1\\|i-j|\leq4}}
       +\sum_{\substack{i\geq-1\\i\geq j-3}}\right)
          2^{2(j+k)/r}x_{i,k}.
\]
Then
\[
 |\mathcal T x|\lesssim\sum_{i,k\geq-1}2^{2(i+k)/r}|x_{i,k}|.
\]
\end{lemma}
\begin{proof}
Take absolute values and interchange the nonnegative sums. Since $2/r>0$,
\[
 \sum_{\substack{j\geq-1\\|i-j|\leq4}}2^{2j/r}
 +\sum_{-1\leq j\leq i+3}2^{2j/r}
 \lesssim2^{2i/r}.
\]
Multiplying by $2^{2k/r}|x_{i,k}|$ and summing proves the claim.
\end{proof}

We next prove the weighted Besov embedding
\eqref{Embedding inequailty}.

\begin{proof}[The proof of \eqref{Embedding inequailty}]
Fix $(t,\eta,p)$ and write $w(y)=W(t,\eta,y,p)$. The polynomial weight
satisfies
\[
 w(y)\lesssim\langle y-z\rangle^\ell w(z)
\]
uniformly in the fixed parameters. For $j\geq0$, choose a Schwartz
reproducing kernel $K_j(y)=2^{2j}K(2^jy)$ whose Fourier multiplier is one
on the support of $\Delta_j$. Then $K_j*\Delta_jf=\Delta_jf$ and
\begin{align*}
 \|w\Delta_jf\|_{L_y^{\infty}}
 &\lesssim\big\||K_j|\langle\cdot\rangle^\ell\big\|_{L_y^{r'}}
                     \|w\Delta_jf\|_{L_y^r}\\
 &\lesssim2^{2j/r}\|w\Delta_jf\|_{L_y^r}.
\end{align*}
For $j=-1$, choose a separate Schwartz multiplier equal to one on the
support of $\chi$; it gives the same bound after adjusting the constant.
Summing the blocks proves \eqref{Embedding inequailty}, first for smooth
functions and then by weighted reconstruction. In particular, the constant
is uniform in time, momentum, and the longitudinal variable.
\end{proof}
The following lemma establishes the boundedness of the dyadic projection operators.
\begin{lemma}\label{ds-bd-lem}
Let $1\leq r \leq \infty$ and $\ell\geq0$. The dyadic blocks and low-frequency cutoffs satisfy
 \begin{equation}\label{bd - w(v)Delta_k}
		\norm{W_p\Delta_j f}_{L_y^r}\lesssim \norm{W_pf}_{L_y^r},\  \norm{W_pS_{j}f}_{L_y^r} \lesssim \norm{W_pf}_{L_y^r}.
	\end{equation}
Taking the $L_{\eta}^r$ norm also gives
	$$  \norm{W_p\Delta_j f}_{L_x^r}\lesssim \norm{W_pf}_{L_x^r}, \norm{W_pS_{j}f}_{L_x^r} \lesssim \norm{W_pf}_{L_x^r}.$$

\end{lemma}
\begin{proof}
For $j\geq0$, take absolute values before comparing weights:
\[
 |W_p(y)\Delta_jf(y)|
 \leq C_\ell\int_{\mathbb R^2}2^{2j}|\psi(2^jz)|\langle z\rangle^\ell
                 W_p(y-z)|f(y-z)|\,dz.
\]
The kernel on the right has uniformly bounded $L^1$ norm for $j\geq0$.
Young's inequality proves the block estimate. The low block uses
$\widetilde\psi$ instead; the low-frequency cutoffs use its rescalings
for $j\geq0$ and vanish for $j\leq-1$. This proves both claims.

The same argument also applies to $V(y)=W_p(y)W_q(y)$, since
\[
 V(y)\lesssim\langle z\rangle^{2\ell}V(y-z).
\]
Thus, uniformly in $p,q,t,\eta$ and $j\geq-1$,
\[
 \|W_pW_q\Delta_jh\|_{L_y^r}
 \lesssim\|W_pW_qh\|_{L_y^r}.
\]
This is the product-weight estimate used in the paraproduct calculation.
\end{proof}

We use the following constant-speed form of the div--curl interaction
estimate; compare \cite{wz-24}. Its $L^1$ formulation is suited to the
absence of longitudinal derivatives in the solution space.
\begin{lemma}[Div--curl interaction estimate]\label{div-curl Lemma}
Let $A,B\geq0$ belong to $C([0,T];L^1(\mathbb R))$ and satisfy
\[
 \partial_tA+a\partial_\eta A=F,\qquad
 \partial_tB+b\partial_\eta B=G
\]
in distributions, where $a,b\in\mathbb R$ are constants and
$F,G\in L^1((0,T)\times\mathbb R)$. Then
\begin{align}
 \int_0^T\!\int_{\mathbb R}|a-b|AB\,d\eta\,dt
 &\leq\|A(0)\|_1\|B(0)\|_1+\|A(T)\|_1\|B(T)\|_1\notag\\
 &\quad+\int_0^T\big(\|A(t)\|_1\|G(t)\|_1
                       +\|B(t)\|_1\|F(t)\|_1\big)\,dt.\label{dc-ine}
\end{align}
\end{lemma}
\begin{proof}
The case $a=b$ is immediate. After interchanging the equations if necessary,
assume $a>b$. For smooth densities, let
\[
 I(t)=\int_{\eta<\zeta}A(t,\eta)B(t,\zeta)\,d\eta\,d\zeta.
\]
Directly calculation yields
\begin{align*}
 I'(t)={}&-(a-b)\int AB\,d\eta+\int\left(\int_{-\infty}^{\eta}A(t,\zeta)\,d\zeta\right)G(t,\eta)\,d\eta\\
 &+\int\left(\int_{\eta}^{\infty}B(t,\zeta)\,d\zeta\right)F(t,\eta)\,d\eta.
\end{align*}
The potential and source terms satisfy
\[
 0\leq I(t)\leq\|A(t)\|_1\|B(t)\|_1,
\]
\[
 \left|I'(t)+(a-b)\int_{\mathbb R}AB\,d\eta\right|
 \leq\|A(t)\|_1\|G(t)\|_1+\|B(t)\|_1\|F(t)\|_1.
\]
Hence
\begin{align*}
 (a-b)\int_0^T\!\int_{\mathbb R}AB\,d\eta\,dt
 &\leq I(0)-I(T)
   +\int_0^T\big(\|A(t)\|_1\|G(t)\|_1
                    +\|B(t)\|_1\|F(t)\|_1\big)\,dt\\
 &\leq\|A(0)\|_1\|B(0)\|_1+\|A(T)\|_1\|B(T)\|_1\\
 &\quad+\int_0^T\big(\|A(t)\|_1\|G(t)\|_1
                    +\|B(t)\|_1\|F(t)\|_1\big)\,dt,
\end{align*}
which gives \eqref{dc-ine}. This completes the proof of Lemma \ref{div-curl Lemma}.
\end{proof}

\noindent {\bf Acknowledgements:}
	SQL was supported by grants from the National Natural Science Foundation of China under contracts 12325107 and 12471217.
	YZ was supported by grant from the National Natural Science Foundation of China under contract 12571231.

\medskip
\noindent\textbf{Data Availability Statement:}
Data sharing is not applicable to this article as no datasets were generated or analysed during the current study.

\noindent\textbf{Conflict of Interest:}
The authors declare that they have no conflict of interest.

\noindent\textbf{AI usage disclosure:}  The authors used AI-assisted tools for writing assistance and to help check the correctness of the mathematical arguments. The authors have carefully reviewed all AI-assisted revisions and take full responsibility for the content, arguments, and conclusions presented in the manuscript.

\end{document}